\documentclass{amsproc}
\usepackage[margin=1.2in,nomarginpar]{geometry}
\usepackage{amssymb}
\usepackage{amsmath}
\usepackage{mathdots}
\usepackage{amsbsy}
\usepackage{amscd}
\usepackage{amsthm}

\usepackage{url}
\usepackage{xcolor}
\usepackage{amsmath,amssymb,amsthm,amsfonts,longtable}
\usepackage[english]{babel}
\usepackage{tikz-cd}
\usetikzlibrary{cd}
\usetikzlibrary{decorations.markings}
\tikzset{negated/.style={
		decoration={markings,
			mark= at position 0.5 with {
				\node[transform shape] (tempnode) {$\times$};
			}
		},
		postaction={decorate}
	}
}
\usepackage{array}
\usepackage[colorlinks,citecolor=red,urlcolor=blue,bookmarks=false,hypertexnames=true]{hyperref} 
\newtheorem{theorem}{Theorem}
\newtheorem{corollary}[theorem]{Corollary}
\newtheorem{lemma}[theorem]{Lemma}
\newtheorem{proposition}[theorem]{Proposition}

\newtheorem{remark}[theorem]{Remark}
\newtheorem{algorithm}[theorem]{Algorithm}
\newtheorem{example}[theorem]{Example}
\newtheorem{definition}{Definition}[subsection]
\newcommand{\Irr}{\textnormal{Irr}}
\newcommand{\cd}{\textnormal{cd}}
\newcommand{\nl}{\textnormal{nl}}
\newcommand{\lin}{\textnormal{lin}}
\newcommand{\FIrr}{\textnormal{FIrr}}
\newcommand{\Core}{\textnormal{Core}}
\newcommand{\gal}{\textnormal{Gal}}

\title[]{Rational Representations and Rational Group Algebra of VZ $p$-groups}
\author{Ram Karan Choudhary}
\address{Indian Institute of Technology, Bhubaneswar, Arugul Campus, Jatni, Khurda-752050, India.}
\email{ramkchoudhary1997@gmail.com}
\author{Sunil Kumar Prajapati$^*$}
\address{Indian Institute of Technology, Bhubaneswar, Arugul Campus, Jatni, Khurda-752050, India.}
\email{skprajapati@iitbbs.ac.in}
\thanks{$^{\textbf{*}}$ Corresponding author.
}
\begin{document}
	\maketitle

	\begin{abstract}
		In this article, we study rational matrix representations of VZ $p$-groups ($p$ is any prime). Utilizing our findings on VZ $p$-groups, we explicitly obtain all inequivalent irreducible rational matrix representations of all $p$-groups of order $\leq p^4$. Furthermore, we establish combinatorial formulas to determine the Wedderburn decompositions of rational group algebras for VZ $p$-groups and all $p$-groups of order $\leq p^4$, ensuring simplicity in the process.
	\end{abstract}
	
	\section{Introduction}
	This paper consistently employs the following notations: $G$ for a finite group, $\Irr(G)$ for the set of all complex irreducible characters of $G$, $\mathbb{F}$ for a field with characteristic $0$, and $p$ for a prime number. In Representation Theory, a challenging and crucial task is to compute all inequivalent irreducible matrix representations of $G$ over $\mathbb{F}$, even for $\mathbb{F}=\mathbb{C}$. In this paper, we deal with $\mathbb{F}=\mathbb{Q}$. For a given character $\chi \in \Irr(G)$, we define $\Omega(\chi)$ as follows:
		\begin{equation*}
		\Omega(\chi) = m_{\mathbb{Q}}(\chi)\sum_{\sigma \in \gal(\mathbb{Q}(\chi) / \mathbb{Q})}^{}\chi^{\sigma},
	\end{equation*}
	where $m_{\mathbb{Q}}(\chi)$ represents the Schur index of $\chi$ over $\mathbb{Q}$. Note that $\Omega(\chi)$ corresponds to the character of an irreducible $\mathbb{Q}$-representation $\rho$ of $G$. Conversely, if $\rho$ is an irreducible $\mathbb{Q}$-representation of $G$, then there exists $\chi \in \Irr(G)$ such that $\Omega(\chi)$ is the character of $\rho$. Generally, obtaining an irreducible $\mathbb{Q}$-representation $\rho$ of $G$ that affords the character $\Omega(\chi)$ is a challenging task. A significant and interesting task is to determine all inequivalent irreducible matrix representations of $G$ over $\mathbb{Q}$ for several reasons. For example, one problem in rationality concerns the realizability of an $\mathbb{F}$-representation of $G$ over its subfields, specially, the realizability of a $\mathbb{C}$-representation of $G$ over $\mathbb{R}$ or $\mathbb{Q}$.\\
   In this article, we study irreducible rational matrix representations for some classes of  $p$-groups.  In Section \ref{sec:Algorithm}, for any $p$-group $G$, we present Algorithm \ref{algorithm} and Algorithm \ref{algorithm2} for computing an irreducible rational matrix representation of $G$ that affords the character $\Omega(\chi)$, where $\chi\in \Irr(G)$. These algorithms rely heavily on results from \cite{CEF, Y, Y1}. To obtain an irreducible rational matrix representation of a finite $p$-group $G$ affording the character $\Omega(\chi)$, where $\chi \in \Irr(G)$, is equivalent to finding a pair $(H, \psi)$, where $H\leq G$ and $\psi\in\Irr(H)$, with some suitable properties (see Algorithm \ref{algorithm} and Algorithm \ref{algorithm2}). 
   We refer to $(H, \psi)$ as a required pair for an irreducible rational matrix representation of $G$ that affords the character $\Omega(\chi)$. Importantly, using a required pair $(H, \psi)$, we can also compute an irreducible complex matrix representation of $G$ which affords the character $\chi$.	
   In Section \ref{sec:VZ $p$-group}, we study required pairs for VZ $p$-group. A group $G$ is called a VZ-group if all its non-linear irreducible characters vanish off the center. VZ-groups have been extensively studied by various researchers in \cite{FM,  MLL2,  MLL, SKP, SKP1}. By using our results on VZ $p$-groups, we explicitly obtain all inequivalent irreducible rational matrix representations of all $p$-groups of order $\leq p^4$ (see Sections \ref{sec:p^3} and \ref{sec:p^4}). 

In parallel, this article delves into an investigation of the Wedderburn decomposition of $\mathbb{Q}G$ with a specific focus on VZ $p$-groups. For a semisimple group algebra $\mathbb{F}G$, the Wedderburn components are matrix algebras over finite extensions of $\mathbb{F}$ in the case of positive characteristic, and Brauer equivalent to cyclotomic algebras in the case of zero characteristic, as per the Brauer-Witt Theorem (see \cite{Yam}). Further, the Wedderburn decomposition of $\mathbb{F}G$ aids in describing the automorphisms group of $\mathbb{F}G$ (see \cite{Herman, Olivieri}) or studying the unit group of the integral group ring $\mathbb{Z}G$ when $\mathbb{F}=\mathbb{Q}$ (see \cite{Jes-Lea, Jes-Rio, Rio-Rui, Rit-Seh}). The Wedderburn decomposition of $\mathbb{Q}G$ haven been extensively studied in \cite{ BG1, BG, BM14, Jes-Lea-Paq, ODRS04, Olt07, PW}. They have used various concepts such as computation of the field of character values, Shoda pairs, numerical representation of cyclotomic algebras, etc, to compute simple components of $\mathbb{Q}G$. We prove Theorem \ref{thm:wedderburn VZ} and Theorem \ref{thm:WedderburnVZ 2-gp}, which provide a combinatorial description for Wedderburn decomposition of rational group algebra of a VZ $p$-group. Our results formulate the computation of the Wedderburn decomposition of a VZ $p$-group $G$ solely based on computing the number of cyclic subgroups of $Z(G)$ and $Z(G)/G'$, which is similar to Perlis-Walker Theorem for an abelian group (see Lemma \ref{Perlis-walker}). 
Indeed, we prove the following theorems. 
\begin{theorem}\label{thm:wedderburn VZ}
	Let $G$ be a finite VZ $p$-group (odd prime $p$). Let $m_1$, $m_2$ and $m_3$ denote the exponents of $G/G'$, $Z(G)$ and $Z(G)/G'$, respectively. Then the Wedderburn decomposition of $\mathbb{Q}G$ is as follows:
	\[\mathbb{Q}G \cong \bigoplus_{d_1|m_1}a_{d_1}\mathbb{Q}(\zeta_{d_1})\bigoplus_{d_2\mid m_2, d_2 \nmid m_3} a_{d_2}M_{|G/Z(G)|^{\frac{1}{2}}}(\mathbb{Q}(\zeta_{d_2}))\bigoplus_{d_2|m_2, d_2|m_3}(a_{d_2}-a_{d_2}')M_{|G/Z(G)|^{\frac{1}{2}}}(\mathbb{Q}(\zeta_{d_2})),\]
	where $a_{d_1}$, $a_{d_2}$ and $a_{d_2}'$ are the number of cyclic subgroups of $G/G'$ of order $d_1$, the number of cyclic subgroups of $Z(G)$ of order $d_2$ and the number of cyclic subgroups of $Z(G)/G'$ of order $d_2$, respectively.
\end{theorem}
\begin{theorem}\label{thm:WedderburnVZ 2-gp}
	Let $G$ be a VZ $2$-group. Let $m_1$, $m_2$ and $m_3$ denote the exponents of $G/G'$, $Z(G)$ and $Z(G)/G'$, respectively. Suppose $k=|\{\chi\in \nl(G) : m_{\mathbb{Q}}(\chi)=2\}|$, and $\mathbb{H}(\mathbb{Q})$ represents the quaternion algebra over $\mathbb{Q}$. Then the Wedderburn decomposition of $\mathbb{Q}G$ is as follows:
	\begin{align*}
		\mathbb{Q}G \cong &\bigoplus_{d_1\mid m_1}a_{d_1}\mathbb{Q}(\zeta_{d_1}) \bigoplus kM_{\frac{1}{2}|G/Z(G)|^{\frac{1}{2}}}(\mathbb{H}(\mathbb{Q})) \bigoplus(a_2-a_2'-k)M_{|G/Z(G)|^{\frac{1}{2}}}(\mathbb{Q})\\
		&\bigoplus_{d_2\mid m_2, d_2 \nmid m_3} a_{d_2}M_{|G/Z(G)|^{\frac{1}{2}}}(\mathbb{Q}(\zeta_{d_2})) \bigoplus_{d_2\mid m_2, d_2 \mid m_3}(a_{d_2}-a_{d_2}')M_{|G/Z(G)|^{\frac{1}{2}}}(\mathbb{Q}(\zeta_{d_2})),
	\end{align*}
	where $a_{d_1}$, $a_{l}$ and $a_{l}'$ $(l \in \{2, d_2\geq 4\})$ are the number of cyclic subgroups of $G/G'$ of order $d_1$, the number of cyclic subgroups of $Z(G)$ of order $l$ and the number of cyclic subgroups of $Z(G)/G'$ of order $l$, respectively.
\end{theorem} 
In subsection \ref{subsec:VZalgebra}, we derive several consequences from the above theorems. Further,  if $G$ is a non-abelian $p$-group of order $p^4$ of maximal class then $G$ has a unique abelian subgroup of index $p$ (see Subsection \ref{subsec:p4class3}). We prove Theorem \ref{thm:wedderburn nonVZp^4}, which formulates the computation of the Wedderburn decomposition of a non-VZ $p$-group $G$ of order $p^4$. 
\begin{theorem}\label{thm:wedderburn nonVZp^4}
	Let $G$ be a non-abelian $p$-group (odd prime $p$) of order $p^4$ of nilpotency class $3$, and let $H$ be its unique abelian subgroup of index $p$. Let $m$ and $m'$ denote the exponents of $H$ and $H/G'$, respectively. Then the Wedderburn decomposition of $\mathbb{Q}G$ is as follows:
	\[\mathbb{Q}G\cong\bigoplus\mathbb{Q}(G/G')\bigoplus_{d|m, d \nmid m'} \frac{a_{d}}{p}M_p(\mathbb{Q}(\zeta_{d}))\bigoplus_{d|m, d|m'}\frac{a_{d}-a_{d}'}{p}M_p(\mathbb{Q}(\zeta_{d})),\]
	where $a_{d}$ and $a_{d}'$ are the number of cyclic subgroups of order $d$ of $H$ and $H/G'$, respectively.
\end{theorem}
In this article, we also provide a brief analysis of primitive central idempotents and their corresponding simple components in the Wedderburn decomposition of the rational group ring of a VZ $p$-group (see Subsection \ref{subsec:pci}).
	
	\section{Notations and Some basic results}\label{section:notation}
	\subsection{Notations}
	For a finite group $G$,
	
	\begin{tabular}{cl}
		$G'$ & the commutator subgroup of $G$\\
		$|S|$ & the cardinality of a set $S$\\
		$\Core_{G}(H)$ & the normal core of $H$ in $G$, for $H\leq G$\\
		$\Irr(G)$ & the set of irreducible complex characters of $G$\\
		$\lin(G)$ & $\{\chi \in \Irr(G) : \chi(1)=1\}$\\
		$\nl(G)$ & $\{\chi \in \Irr(G) : \chi(1) \neq 1\}$\\
		$\FIrr(G)$ & the set of faithful irreducible complex characters of $G$\\
		$\Irr^{(m)}(G)$ & $\{\chi \in \Irr(G) : \chi(1)=m\}$\\
		$\cd(G)$ & $\{ \chi(1) : \chi \in \Irr(G) \}$\\
		$\Irr_{\mathbb{Q}}(G)$ & the set of irreducible rational characters of $G$\\
		$\Irr_{\mathbb{Q}}^{(m)}(G)$ & $\{\chi \in \Irr_{\mathbb{Q}}(G) : \chi(1)=m\}$\\
		$\mathbb{F}(\chi)$ &  the field obtained by adjoining the values $\{\chi(g) : g\in G\}$ to the field $\mathbb{F}$, for some $\chi\in \Irr(G)$ \\
		$m_\mathbb{Q}(\chi)$ & the Schur index of $\chi \in \Irr(G)$ over $\mathbb{Q}$\\
		$\Omega(\chi)$ & $m_{\mathbb{Q}}(\chi)\sum_{\sigma \in \gal(\mathbb{Q}(\chi) / \mathbb{Q})}^{}\chi^{\sigma}$, for $\chi \in \Irr(G)$\\
		$\ker(\chi)$ & $\{g \in G: \chi(g)=\chi(1)\}$, for $\chi \in \Irr(G)$\\
		$\Irr(G|N)$ & $\{\chi \in \Irr(G) : N \nsubseteq \ker(\chi)\}$, where $N \trianglelefteq G$\\
		$\psi^G$ & the induced character of $\psi$ to $G$, where $\psi$ is a character of $H$ for some $H \leq G$\\
		$\Psi^G$ & the induced representation of $\Psi$ to $G$, where $\Psi$ is a representation of $H$ for some $H \leq G$\\
		$\chi \downarrow_H$ & the restriction of a character $\chi$ of $G$ on $H$, where $H \leq G$\\
		$\mathbb{F}G$ & the group ring (algebra) of $G$ with coefficients in $\mathbb{F}$\\
		$M_{n}(D)$ & a full matrix ring of order $n$ over the skewfield $D$\\
		$Z(B)$ & the center of an algebraic structure $B$\\
		$\phi(n)$ & the Euler phi function\\
		$\zeta_m$ & an $m$-th primitive root of unity
	\end{tabular}
	
	\subsection{Basic results}
	In this subsection, we discuss some basic concepts and results, which we  use frequently throughout the article. 
	
Let $G$ be a finite group and $n=|G|$. Consider $\mathbb{Q}(\zeta_n)$ the $n$-th cyclotomic field obtained by adjoining a primitive $n$-th root of unity to $\mathbb{Q}$. Let $\chi \in \Irr(G)$. If $\sigma \in \gal(\mathbb{Q}(\zeta_n)/\mathbb{Q})$, then define the function	$\chi^{\sigma}: G \longleftrightarrow \mathbb{C}$ as $\chi^{\sigma}(g)=\sigma(\chi(g))$ for $g\in G$. It is easy to observe that $\chi^{\sigma}\in \Irr(G)$ and hence, $\gal(\mathbb{Q}(\zeta_n)/\mathbb{Q})$ acts on $\Irr(G)$. Note that if $\mathbb{Q}(\chi)$ is a finite degree Galois
extension of $\mathbb{Q}$ and the Galois group $\gal(\mathbb{Q}(\chi)/\mathbb{Q})$ is abelian. It is easy to see that $\gal(\mathbb{Q}(\chi)/\mathbb{Q})$ also acts on $\Irr(G)$, given by $\sigma \cdot \chi := \chi^{\sigma}$. Under the above set up we have the following lemma.
	\begin{lemma}\textnormal{\cite[Lemma 9.17]{I}}\label{SC}
		Let $E(\chi)$ denote the Galois conjugacy class of $\chi\in \Irr(G)$ under the action of $\gal(\mathbb{Q}(\zeta_n)/\mathbb{Q})$. Then
		\[|E(\chi)| = [\mathbb{Q}(\chi) : \mathbb{Q}]. \]
	\end{lemma} 
	\begin{definition}
		Let $\chi, \psi \in \Irr(G)$. We say that $\chi$ and $\psi$ are Galois conjugates over $\mathbb{Q}$ if $\mathbb{Q}(\chi) = \mathbb{Q}(\psi)$, and there exists $\sigma \in$ $\gal(\mathbb{Q}(\chi) / \mathbb{Q})$ such that $\chi^\sigma = \psi$.
	\end{definition}

 In \cite{PW}, Perlis and Walker studied the group ring of a finite abelian group $G$ over the field of rational numbers and proved the following result.
	\begin{lemma}[Perlis-Walker Theorem]\label{Perlis-walker}
		Let $G$ be a finite abelian group of exponent $m$. Then the Wedderburn decomposition of $\mathbb{Q}G$ is as follows:
		\[\mathbb{Q}G \cong \bigoplus_{d|m} a_d \mathbb{Q}(\zeta_d),\]
		where $a_d$ is equal to the number of cyclic subgroups of $G$ of order $d$.
	\end{lemma} 
	\begin{remark}\label{Rational Abelian1}
		Let $G$ be a finite abelian group of exponent $m$. Then by Lemma \ref{Perlis-walker}, we have  
		\begin{itemize}
			\item[1.] the number of rational irreducible representations of $G$ of degree $\phi(d)$ is equal to $a_d$, and
			\item[2.] the total number of rational irreducible representations of $G$ is equal to $\sum_{d|m}a_d$.
		\end{itemize}
	\end{remark}
	
\noindent Let $\mathbb{K}$ be an arbitrary field with characteristic zero and  $\mathbb{K}^*$ be the algebraic closure of $\mathbb{K}$. Let $U$ be an irreducible $\mathbb{K}^*$-representation of $G$ with character $\chi$. The \emph{Schur index} of $U$ with respect to $\mathbb{K}$ is defined as
	\[m_\mathbb{K}(U) = \textnormal{Min} [\mathbb{L} : \mathbb{K}(\chi)],\]
	the minimum being taken over all fields $\mathbb{L}$ in which $U$ is realizable. Note that  $m_{\mathbb{K}}(\chi)=m_{\mathbb{K}}(U)$. 
Reiner \cite{IR} characterized the simple component of the Wedderburn decomposition of $\mathbb{K}G$ and proved the following result.
	\begin{lemma}\cite[Theorem 3]{IR} \label{Reiner}
		Let $T$ be an irreducible $\mathbb{K}$-representation of $G$, and extend $T$ (by linearity) to a $\mathbb{K}$-representation of $\mathbb{K}G$. Set
		\[A=\{T(x) : x\in \mathbb{K}G\}.\]
		Then $A$ is simple algebra over $\mathbb{K}$, and we may write  $A= M_{n}(D)$, where $D$ is a division ring. Further,
		\begin{center}
			$Z(D)\cong \mathbb{K}(\chi)$ and $[D : Z(D)]=(m_\mathbb{K}(U_i))^2~(1\leq i \leq k),$
		\end{center}
		where $U_i$ are irreducible $K^{*}$-representations of $G$ such that $T=m_{\mathbb{K}(U_i)}\bigoplus_{i=1}^k U_i$ as a $K^{*}$-representation.
	\end{lemma}

In this article, we use the the classification of $p$-groups of order $\leq p^4$ (odd prime $p$) provided in \cite{RJ}, which is based on the isoclinism concept. 
 
\begin{definition} Two  finite groups $G$ and $H$ are said to be isoclinic if
	there exist isomorphisms $\theta : G/Z(G)\longrightarrow H/Z(H)$ and $\phi: G{}'\longrightarrow H{}'$
	such that the following diagram is
	commutative:
	\begin{equation*}
		\begin{tikzcd}
			G/Z(G)\times G/Z(G) \arrow{d}{\theta \times \theta} \arrow{r}{a_G}
			& G{}' \arrow{d}{\phi} \\
			H/Z(H)\times H/Z(H) \arrow{r}{a_H}
			& H{}',
		\end{tikzcd}
	\end{equation*}
	where $a_G(g_1Z(G), g_2Z(G)) = [g_1,g_2]$, for $g_1,g_2\in G$, and $a_H(h_1Z(H), h_2Z(H)) = [h_1,h_2]$ for $h_1,h_2\in H$.
\end{definition}
The resulting pair $(\theta, \phi)$ is called an \emph{isoclinism} of $G$ onto $H$.
Isoclinism was first introduced by Hall \cite{PH} for a classification of $p$-groups. It is a generalization of the concept of isomorphism between two groups.
 It is well-known that two isoclinic nilpotent groups have the same nilpotency class. 
Now, we end this subsection by quoting the following lemma.
\begin{lemma}\label{lemma:isoclinicCharacter}\textnormal{\cite[Theorem 3.2]{JT}} Let $G$ and $H$ be isoclinic groups. Then $|H||\Irr^{(k)}(G)| = |G||\Irr^{(k)}(H)|$.
\end{lemma}

	
	
%
%
	
	
	\section{Algorithm} \label{sec:Algorithm}
	This section outlines the algorithm for computing irreducible rational matrix representations of $p$-groups.
		For $\chi \in \Irr(G)$, there exists a unique irreducible $\mathbb{Q}$-representation $\rho$ of $G$ such that $\chi$ occurs as an irreducible constituent of $\rho \otimes_{\mathbb{Q}}\mathbb{F}$ with multiplicity 1, where $\mathbb{F}$ is a splitting field of $G$. Therefore, the distinct Galois conjugacy classes give the distinct irreducible rational representations of $G$.
	\begin{lemma} \textnormal{\cite[Proposition 1]{Y}}\label{lemma:YamadaLinear}
		Let $\psi \in \lin(G)$ and $N = \ker(\psi)$ with $n = [G : N]$. Suppose $G = \cup_{i = 0}^{n-1}Ny^i$. Then
		\[\psi(xy^i) = \zeta_n^i,~  (0 \leq i < n;~ x \in N).\]
		Now, let $f(X) = X^s - a_{s-1}X^{s-1} - \cdots -a_1X - a_0$ be the irreducible polynomial over $\mathbb{Q}$ such that $f(\zeta_n) = 0$, where $s = \phi(n)$ and
		\[\Psi(xy^i) = \left(\begin{array}{ccccc}
			0 & 1 & 0 & \cdots & 0\\
			0 & 0 & 1 & \cdots & 0\\
			\vdots & \vdots & \vdots & \ddots \\
			0 & 0 & \cdots & 0 & 1\\
			a_0 & a_1 & \cdots & \cdots & a_{s-1}
		\end{array}\right)^i,~ (0 \leq i < n;~ x \in N).\]
		Then $\Psi$ is an irreducible $\mathbb{Q}$-representation of $G$, whose character is $\Omega(\psi)$.
	\end{lemma}	 
	\begin{lemma} \textnormal{\cite[Proposition 3]{Y}}\label{lemma:Yamada}
		Let $H$ be a subgroup of $G$ and $\psi \in \Irr(H)$ such that $\psi^G \in\Irr(G)$. Then $m_\mathbb{Q}(\psi^G)$ divides $m_\mathbb{Q}(\psi)[\mathbb{Q}(\psi) : \mathbb{Q}(\psi^G)]$. Furthermore, the induced character $\Omega(\psi)^G$ of $G$ is a character of an irreducible $\mathbb{Q}$-representation of $G$, if and only if
		\[m_\mathbb{Q}(\psi^G) = m_\mathbb{Q}(\psi)[\mathbb{Q}(\psi) : \mathbb{Q}(\psi^G)].\]
		In this case, $\Omega(\psi)^G = \Omega(\psi^G)$.
	\end{lemma}
	\begin{lemma} \textnormal{\cite[Theorem 1]{CEF}}\label{lemma:Ford}
		Let $G$ be a $p$-group and $\chi\in Irr(G)$. Then one of the following holds:
		\begin{enumerate}
		\item There exists a subgroup $H$ of $G$ and $\psi\in \lin(H)$ such that
	$\psi^G = \chi$ and $\mathbb{Q}(\psi) = \mathbb{Q}(\chi)$.
	\item $p=2$ and there exist subgroups $H< K$ in $G$ with $[K : H] = 2 $ and $\lambda \in \lin(H)$ such that $\lambda^K=\phi$, $[\mathbb{Q}(\lambda) : \mathbb{Q}(\phi) ]=2$, $\phi^G=\chi$ and $\mathbb{Q}(\phi) =\mathbb{Q}(\chi)$. 
\end{enumerate}	
	\end{lemma}
\begin{lemma}\textnormal{\cite[Corollary 10.14]{I}}\label{lemma:schurindexpgroup}
		Let $G$ be a $p$-group and $\chi \in \Irr(G)$. If $p$ is an odd prime, then  $m_\mathbb{Q}(\chi) = 1$, otherwise $m_\mathbb{Q}(\chi)\in \{1,2\}$.
	\end{lemma}
     
Let $G$ be a $p$-group (odd prime $p$) and $\chi \in \Irr(G)$. According to Lemma \ref{lemma:schurindexpgroup}, $m_{\mathbb{Q}}(\chi)=1$. By Lemma \ref{lemma:Ford}, there exists a subgroup $H$ of $G$ with $\psi \in \lin(H)$ such that $\psi^G = \chi$ and $\mathbb{Q}(\psi) = \mathbb{Q}(\chi)$. Therefore from Lemma \ref{lemma:Yamada}, we have $\Omega(\chi)=\Omega(\psi)^G$. Now by using Lemma \ref{lemma:YamadaLinear}, compute an irreducible matrix representation $\Psi$ of $H$ over $\mathbb{Q}$ which affords the character $\Omega(\psi)$. Then $\Psi^G$ is an irreducible $\mathbb{Q}$-representation of $G$ that affords the character $\Omega(\psi)^G = \Omega(\psi^G) = \Omega(\chi)$.
In summary, an algorithm to find out an irreducible rational matrix representation of a $p$-group $G$   (odd prime $p$) can be outlined as follows.	 	 
	 \begin{algorithm} \label{algorithm}
	 	Input: An irreducible complex character $\chi$ of a finite $p$-group $G$ (odd prime $p$).
	 	\begin{itemize}
	 			\item[(1)] Find a pair $(H, \psi)$, where $H \leq G$ and $\psi \in \lin(H)$ such that $\psi^G = \chi$ and $\mathbb{Q}(\psi) = \mathbb{Q}(\chi)$.
	 			\item[(2)] Find an irreducible $\mathbb{Q}$-representation $\Psi$ of $H$ which affords the character $\Omega(\psi)$.
	 			\item[(3)] Induce $\Psi$ to $G$.
	 	\end{itemize}
	 	Output: $\Psi^G$, an irreducible $\mathbb{Q}$-representation of $G$ whose character is $\Omega(\chi)$.
	 \end{algorithm}
 
 \begin{remark}\label{remark:requiredpair}
 \textnormal{Obtaining an irreducible rational matrix representation of a finite $p$-group $G$ (odd prime $p$) affording the character $\Omega(\chi)$, where $\chi \in \Irr(G)$, is equivalent to finding a pair $(H, \psi)$, where $H$ is a subgroup of $G$ and $\psi \in \lin(H)$, satisfying $\psi^G = \chi$ and $\mathbb{Q}(\psi) = \mathbb{Q}(\chi)$. We refer to this pair $(H, \psi)$ as a required pair for an irreducible rational matrix representation of $G$ which affords the character $\Omega(\chi)$.}
 \end{remark}
	

Now, we describe the algorithm for computing irreducible rational matrix representations of $2$-groups.

	\begin{lemma}\cite[Theorem 2.12]{Y1}\label{lemma:existanceCQDSd}
		Let $G$ be a $2$-group and $\chi \in \Irr(G)$. Then there exists a pair $(H, \psi)$ such that $H \leq G$, $\psi \in \Irr(H)$, $\psi^G = \chi$, $\mathbb{Q}(\chi) = \mathbb{Q}(\psi)$, and one of the following holds:
		\begin{enumerate}
		\item $H/\ker(\psi)\cong Q_n$ $(n\geq 2)$, $m_{\mathbb{Q}}(\chi)=2$, $\mathbb{Q}(\chi)=\mathbb{Q}(\zeta_{2^n}+\zeta^{-1}_{2^n}),$
		\item $H/\ker(\psi)\cong D_n$ $(n\geq 3)$, $m_{\mathbb{Q}}(\chi)=1$, $\mathbb{Q}(\chi)=\mathbb{Q}(\zeta_{2^n}+\zeta_{2^n}^{-1}),$
		\item $H/\ker(\psi)\cong SD_n$ $(n\geq 3)$, $m_{\mathbb{Q}}(\chi)=1$, $\mathbb{Q}(\chi)=\mathbb{Q}(\zeta_{2^n}-\zeta_{2^n}^{-1}),$
		\item $H/\ker(\psi)\cong C_n$ $(n\geq 0)$, $m_{\mathbb{Q}}(\chi)=1$, $\mathbb{Q}(\chi)=\mathbb{Q}(\zeta_{2^n}),$
		\end{enumerate} 
		where $Q_n$, $D_n$ and $SD_n$ are, respectively the generalized quaternion, dihedral and semidihedral group of order $2^{n+1}$, and $C_n$ is the cyclic group of order $2^n$.
	\end{lemma}
	\begin{lemma}\cite[Example 7 and Proposition 8]{Y}
		\begin{itemize}\label{lemma:rationalQDSd}
			\item[(a)] Let $G = Q_n = \langle a, b : a^{2^n} = 1, b^2=a^{2^{n-1}}, bab^{-1}=a^{-1} \rangle$ be the generalized quaternion group of order $2^{n+1}$, and let $\chi \in \FIrr(G)$. Then $\chi = \psi^G$ where $H = \langle a \rangle$ and $\psi \in \lin(H)$ such that $\psi(a)=\zeta_{2^n}$, $m_{\mathbb{Q}}(\chi)=2$ and $\Omega(\chi)=\Omega(\psi)^G$.
			\item[(b)] Let $G = D_n = \langle a, b : a^{2^n} = b^2=1, bab^{-1}=a^{-1} \rangle$ (respectively $G = SD_n = \langle a, b : a^{2^n} = b^2=1, bab^{-1}=a^{2^{n-1}-1} \rangle$) be the dihedral group (respectively semi-dihedral group) of order $2^{n+1}$, and let $\chi \in \FIrr(G)$. Then $\Omega(\chi)=\Omega (\psi)^G$, where $H = \langle a^{2^{n-1}}, b \rangle$ and $\psi \in \lin(H)$ such that $\psi(a^{2^{n-1}}) = -1$, $\psi(b)=1$.
		\end{itemize}
	\end{lemma}
\begin{remark}\label{rem:2groupcases}\textnormal{Let $G$ be a $2$-group and let $\chi \in \nl(G)$. By Lemma \ref{lemma:existanceCQDSd}, there exists a pair $(H, \psi)$, with $H \leq G$ and $\psi \in \Irr(H)$, satisfying the following properties: $\psi^G=\chi$, $\mathbb{Q}(\chi)=\mathbb{Q}(\psi)$, and $H/\ker(\psi)$ is isomorphic to one of the following groups: cyclic group, generalized quaternion group, dihedral group, or semi-dihedral group. We define $\bar{\psi} \in \FIrr(H/\ker(\psi))$ such that $\bar{\psi}(h\ker(\psi))=\psi(h)$ for all $h \in H$. 
Now we have two cases.\\
{\bf Case 1 ( $H/\ker(\psi)\cong C_n$).}  In this case $\psi \in \lin(H)$ and hence by using Lemma \ref{lemma:YamadaLinear}, we get an irreducible rational matrix representation $\Psi$ of $H$ which affords the character $\Omega(\psi)$. Then by Lemma \ref{lemma:Yamada}, $\Psi^G$ is an irreducible rational matrix representation of $G$ which affords the character $\Omega(\chi)$. \\
{\bf Case 2 ($H/\ker(\psi)\cong Q_n$, or $D_n$, or $SD_n$ for some $n \in \mathbb{N}$).} In this case, since $\bar{\psi} \in \FIrr(H/\ker(\psi))$, by Lemma \ref{lemma:rationalQDSd} and Lemma \ref{lemma:YamadaLinear}, there exists an irreducible rational matrix representation of $H/\ker(\psi)$ which affords the character $\Omega(\bar{\psi})$. Indeed, we get an irreducible rational matrix representation $\Psi$ of $H$ which affords the character $\Omega(\psi)$. Then again by Lemma \ref{lemma:Yamada}, $\Psi^G$ is an irreducible rational matrix representation of $G$ which affords the character $\Omega(\chi)$.}
\end{remark}
In view of Remark \ref{rem:2groupcases}, an algorithm to find out an irreducible rational matrix representation of a $2$-group $G$ can be outlined as follows.	 
	\begin{algorithm} \label{algorithm2}
		Input: An irreducible complex character $\chi$ of a finite $2$-group $G$.
		\begin{itemize}
			\item[(1)] Find a pair $(H, \psi)$, where $H \leq G$ and $\psi \in \Irr(H)$ such that $\psi^G=\chi$, $\mathbb{Q}(\chi)=\mathbb{Q}(\psi)$, and $H/\ker(\psi)$ is cyclic, generalized quaternion, dihedral, or semi-dihedral group.
			\item[(2)] Find an irreducible $\mathbb{Q}$-representation $\Psi$ of $H$ which affords the character $\Omega(\psi)$.
			\item[(3)] Induce $\Psi$ to $G$.
		\end{itemize}
		Output: $\Psi^G$, an irreducible $\mathbb{Q}$-representation of $G$ whose character is $\Omega(\chi)$.\\
Here, we call such a pair $(H, \psi)$ as a \textit{required pair} for an irreducible rational matrix representation of $G$ which affords the character $\Omega(\chi)$.
	\end{algorithm}

	 \section{VZ $p$-group}\label{sec:VZ $p$-group}
	 
	 A group $G$ is called a VZ-group if all its non-linear complex irreducible characters vanish off center (see \cite{MLL}). 
	 In this case, $G' \subseteq Z(G)$ and hence the nilpotency class of $G$ is $2$. The character degree set is given by  $\cd(G) = \{1, |G/Z(G)|^{\frac{1}{2}}\}$. Furthermore, $G$ has $|Z(G)| - |Z(G)/G'|$ many inequivalent nonlinear irreducible complex characters, and there is a one-to-one correspondence between the sets $\nl(G)$ and $\Irr(Z(G) | G')$ (see \cite[Subsection 3.1]{SKP1}). For any $\mu \in \Irr(Z(G) | G')$, the corresponding $\chi_\mu \in \nl(G)$ is defined as follows:
	  \begin{equation}\label{VZ}
	  	\chi_\mu(g) = \begin{cases}
	  		|G/Z(G)|^{\frac{1}{2}}\mu(g)  &\quad \text{ if } g \in Z(G),\\
	  		0            &\quad \text{ otherwise. }
	  	\end{cases}
	  \end{equation}
	  
\noindent Observe that, being a nilpotency class $2$ group, $G$ can be written as a direct product of its Sylow subgroups. Since $\cd(G) = \{1, |G/Z(G)|^{\frac{1}{2}}\}$, all the Sylow subgroups of $G$, except one, are abelian.
	  \subsection{Rational representations of VZ $p$-groups.} \label{subsec:rrVZpgroups} Let $G$ be a VZ-group, and let $\chi_\mu \in \nl(G)$ (as defined in \eqref{VZ}). 	  
 Since $G$ is monomial, there exists a subgroup $H$ of $G$ with index $|G/Z(G)|^{\frac{1}{2}}$, and $\psi \in \lin(H)$, satisfying $\psi^G = \chi_\mu$. We will now present the following results that provide a description of $H$ and $\psi$.

	 \begin{proposition} \label{prop:VZreqpair}
	 	Let $G$ be a VZ-group. Suppose $H$ is a subgroup of $G$ with index $|G/Z(G)|^{\frac{1}{2}}$ and $\psi \in \lin(H)$. Then $\psi^G = \chi_\mu \in \nl(G)$ (as defined in \eqref{VZ}) if and only if 
	 	$H$ is normal in $G$ such that $Z(G)\subset H$
	 		and $\psi \downarrow_{Z(G)} = \mu$ with $\mu\in \Irr(Z(G)|G{}')$.
	 \end{proposition}
	 \begin{proof}
	 	Let $\psi \in \lin(H)$ such that $\psi^G = \chi_\mu$. Let $T$ be a set of right coset representatives of $H$ in $G$. Then for $g\in G$, we have $\psi^G(g) = \sum_{g_i\in T} {{\psi}^{\circ}}(g_igg_i^{-1})$, where ${\psi}^{\circ}$ is defined by  $\psi^\circ(x)=\psi(x)$ if $x \in H$ and  $\psi^{\circ}(x)=0$ if $x\notin H$. Now, for $z\in Z(G)$, we get $\psi^G(z) = |G/Z(G)|^{\frac{1}{2}}{\psi}^\circ(z)$ for $z \in Z(G)$ and since $\psi^G = \chi_\mu$, we obtain ${\psi}^\circ(z) = \mu(z)=\psi(z)$. This implies that $Z(G) \subseteq H$ and $\psi \downarrow_{Z(G)} = \mu$. Since $G{}'\subset Z(G)$,  $H$ is normal in $G$.\\
	    Conversely, assume $H$ is a subgroup of $G$ with index $|G/Z(G)|^{\frac{1}{2}}$, $\psi \in \lin(H)$, and $Z(G) \subseteq H$ with $\psi \downarrow_{Z(G)} = \mu$, where $\mu\in \Irr(Z(G)|G{}')$. Claim: $\psi^G \in \nl(G)$. On the contrary, suppose that $\psi^G \notin \nl(G)$, then $\psi^G$ must be a sum of some linear characters of $G$. Hence $G' \subseteq \ker(\psi^G)$, which is a contradiction as $\psi \downarrow_{Z(G)} = \mu$ where $\mu \in \Irr(Z(G) | G')$. This proves the claim.
	 \end{proof}
 
We prove Proposition \ref{prop:reqpairVZ2-gp} which describes a required pair of a VZ $2$-group.
 	\begin{proposition} \label{prop:reqpairVZ2-gp}
 		Let $G$ be a VZ $2$-group and $\chi \in \nl(G)$. Consider $(H, \psi)$ as a required pair for an irreducible rational matrix representation of $G$ which affords the character $\Omega(\chi)$. Then $H/\ker(\psi)$ is isomorphic to one of the following groups: cyclic group, quaternion group of order $8$ denoted as $Q_8$, or dihedral group of order $8$ denoted as $D_8$.
 	\end{proposition}
 	\begin{proof}
 		Since $H$  is a monomial group, there exists a subgroup $K$ of $H$ such that
 	$\lambda^H = \psi$, where  $\lambda \in \lin(K)$. Consequently, we have $\lambda^G = \chi$. From Proposition \ref{prop:VZreqpair}, it follows that $Z(G) \leq K$, which implies that $G' \leq Z(G) \leq H$. It is worth noting that $H' \leq Z(H)$ and therefore $(H/\ker(\psi))' \leq Z(H/\ker(\psi))$. Hence $H/\ker(\psi)$ must be isomorphic to one of the groups: cyclic, $Q_8$, or $D_8$.
 	\end{proof}
 	Let $G$ be a VZ 2-group and $\chi \in \Irr(G)$. Consider a required pair $(H, \psi)$ for a rational representation of $G$ that affords the character $\Omega(\chi)$. Then by by Proposition \ref{prop:reqpairVZ2-gp}, $H/\ker(\psi)$ is isomorphic to one of the following: cyclic, $Q_8$, or $D_8$.\\
{\bf Case 1 ($H/\ker(\psi)$ is cyclic).} See Case 1 of Remark \ref{rem:2groupcases}, to get an irreducible rational matrix representation of $G$ which affords the character $\Omega(\chi)$.  \\ 	
{\bf Case 2 ($H/\ker(\psi) \cong D_8$).} Suppose $H/\ker(\psi) = \langle a, b : a^4 = b^2=1, bab^{-1}=a^{-1} \rangle$ (i.e. $\psi\in \nl(H)$). Then $\bar{\psi}\in \FIrr(H/\ker(\psi))$ 
and $\bar{\lambda}\in \lin(K/\ker(\psi))$ such that $\bar{\lambda}^{H/\ker(\psi)}=\bar{\psi}$ and $\mathbb{Q}(\bar{\lambda})=\mathbb{Q}(\bar{\psi})=\mathbb{Q}$, where $K/\ker(\psi)=\langle a^2, b\rangle$ and $\bar{\lambda}$ is defined as $\bar{\lambda}(a^2) = -1$, $\bar{\lambda}(b)=1$. This shows that $(K/\ker{\psi}, \bar{\lambda})$ is a required pair for the rational representation of $H/\ker(\psi)$ which affords the character $\Omega(\bar{\psi})$. This implies that $(K, \lambda)$ is also a required pair for the rational representation of $G$ which affords the character $\Omega(\chi)$. Note that $\mathbb{Q}(\lambda)=\mathbb{Q}(\chi)=\mathbb{Q}$, $[G:H]=\frac{1}{2}|G/Z(G)|^{\frac{1}{2}} $ and  $[G:K]=|G/Z(G)|^{\frac{1}{2}}$.\\
{\bf Case 3 ($H/\ker(\psi) \cong Q_8$).}	 Suppose $H/\ker(\psi)= \langle a, b : a^4 = b^4=1, bab^{-1}=a^{-1} \rangle$ (i.e. $\psi\in \nl(H)$). Then $\bar{\psi}\in \FIrr(H/\ker(\psi))$ 
and $\bar{\lambda}\in \lin(K/\ker(\psi))$ such that $\bar{\lambda}^{H/\ker(\psi)}=\bar{\psi}$ and $[\mathbb{Q}(\bar{\lambda}):\mathbb{Q}(\bar{\psi})]=2$, where $K/\ker(\psi)=\langle a \rangle$  and $\bar{\lambda}$ is defined as $\bar{\lambda}(a) = \zeta_{4}$. Since $m_{\mathbb{Q}}(\bar{\psi})=2$, from Lemma \ref{lemma:Yamada}, we get $\Omega(\bar{\lambda})^{H/\ker(\psi)}=\Omega(\bar{\psi})$. 
This implies that there exists a subgroup $K$ of $G$ and $\lambda\in \lin(K)$ such that $\Omega(\lambda)^G=\Omega(\chi)$. Note that $[\mathbb{Q}(\lambda): \mathbb{Q}(\chi)]=2$ , $[G:H]=\frac{1}{2}|G/Z(G)|^{\frac{1}{2}} $ and  $[G:K]=|G/Z(G)|^{\frac{1}{2}}$.
 	\begin{remark}\label{remark:requiredpair2}
 		\textnormal{In view of the above discussion and Algorithm \ref{algorithm}, to obtain an irreducible rational matrix representation of a VZ $p$-group $G$ ($p$ is any prime) which affords the character $\Omega(\chi)$, where $\chi \in \Irr(G)$, we need to do the following: 
 		\begin{enumerate}
 		\item If $m_{\mathbb{Q}}(\chi)=1$, then find out  $H \leq G$ and $\psi \in \lin(H)$ such that $\psi^G = \chi$, $\mathbb{Q}(\psi) = \mathbb{Q}(\chi)$.
 		\item If $m_{\mathbb{Q}}(\chi)=2$, then find out $H \leq G$ and $\psi \in \lin(H)$ such that $\psi^G = \chi$, $[\mathbb{Q}(\psi) : \mathbb{Q}(\chi)]=2$.
 		\end{enumerate}
 		We call such a pair $(H, \psi)$ as a special required pair for an irreducible rational matrix representation of a VZ $p$-group $G$ which affords the character $\Omega(\chi)$, where $\chi \in \Irr(G)$}. 
 \end{remark}
	 
	 
Now, we prove Lemma \ref{lemma:fieldofcharacters} which provides a description of the character fields, which is useful to obtain a special required pair of a VZ $p$-group.
\begin{lemma}\label{lemma:fieldofcharacters}
	Let $G$ be a VZ-group and let $\chi_\mu \in \nl(G)$ (as defined in \eqref{VZ}). Consider a subgroup $H$ of $G$ with index $|G/Z(G)|^{\frac{1}{2}}$ and $\psi_\mu \in \lin(H)$ such that $\psi_\mu^G = \chi_\mu$. Then $\mathbb{Q}(\psi_\mu) = \mathbb{Q}(\chi_\mu)$ if and only if $|\ker(\psi_\mu)/\ker(\mu)|=|G/Z(G)|^{\frac{1}{2}}$, and $|\mathbb{Q}(\psi_\mu) : \mathbb{Q}(\chi_\mu)| = 2$ if and only if $|\ker(\psi_\mu)/\ker(\mu)|=\frac{1}{2}|G/Z(G)|^{\frac{1}{2}}$.
\end{lemma}
\begin{proof} By Proposition \ref{prop:VZreqpair}, $\psi_{\mu} \downarrow_{Z(G)} = \mu$ with $\mu\in \Irr(Z(G)|G{}')$ and 
 $\mathbb{Q}(\chi_\mu)=\mathbb{Q}(\mu)$. 
Observe that
	\begin{align*}
		\mathbb{Q}(\psi_\mu) = \mathbb{Q}(\chi_\mu)=\mathbb{Q}(\mu) & \iff \mathbb{Q}(\zeta_{|H/\ker(\psi_\mu)|}) = \mathbb{Q}(\zeta_{|Z(G)/\ker(\mu)|})\\
		& \iff |H/\ker(\psi_\mu)| = |Z(G)/\ker(\mu)|\\
		& \iff |\ker(\psi_\mu)| = |H/Z(G)| |\ker(\mu)| \\
		& \iff |\ker(\psi_\mu)| = |G/Z(G)|^{\frac{1}{2}} |\ker(\mu)|.
	\end{align*}
Again,
 \begin{align*}
 	[\mathbb{Q}(\psi_\mu) : \mathbb{Q}(\chi_\mu)]=[\mathbb{Q}(\psi_\mu) : \mathbb{Q}(\mu)]= 2 & \iff |\mathbb{Q}(\zeta_{|H/\ker(\psi_\mu)|}) : \mathbb{Q}(\zeta_{|Z(G)/\ker(\mu)|})|=2\\
 	& \iff |H/\ker(\psi_\mu)| = 2|Z(G)/\ker(\mu)|\\
 	& \iff |\ker(\psi_\mu)| = \frac{1}{2}|H/Z(G)| |\ker(\mu)| \\
 	& \iff |\ker(\psi_\mu)| = \frac{1}{2}|G/Z(G)|^{\frac{1}{2}} |\ker(\mu)|.
\end{align*}
This completes the proof.
\end{proof}

 \begin{corollary}\label{coro:VZreqpairabelian}
 	Let $G$ be a VZ-group. Suppose $H$ is a subgroup of $G$ with index $|G/Z(G)|^{\frac{1}{2}}$ and $\psi \in \lin(H)$ such that $\psi^G \in \nl(G)$. If one of the following satisfies: 
 	\begin{itemize}
 		\item[(a)] $\cd(G)=\{1, p\}$,
 		\item[(b)] $|G'|=p$,
 	\end{itemize}
 	 then $H$ is abelian. 
 \end{corollary}
\begin{proof}
	Let $H$ be a subgroup $G$ such that index $|G/Z(G)|^{\frac{1}{2}}$ and $\psi \in \lin(H)$ with $\psi^G \in \nl(G)$. Then by Proposition \ref{prop:VZreqpair}, $Z(G) \subseteq H$ and $\psi \downarrow_{Z(G)} = \mu$.\\
	Case (a): Suppose $\cd(G)=\{1, p\}$. This implies $|G/Z(G)|=p^2$ and $|H/Z(G)|=p$. This shows that $H$ being abelian.\\
	Case (b): Suppose $|G'|=p$. In this case $|H{}'|\in \{1, p\}$. Since $H' \subseteq \ker(\psi^G) = \Core_G(\ker(\psi))$ and $\psi^G\in \nl(G)$, we get $|H{}'|=1$. Thus $H$ is abelian.
\end{proof}

As we know that, for a VZ-group $G$, $G'$ is an elementary abelian subgroup. Then by Corollary \ref{coro:VZreqpairabelian}, we get the following result.
\begin{corollary} \label{cor:reqpairVZcyclic}
	Let $G$ be a VZ-group. Suppose $H$ is a subgroup of $G$ with index $|G/Z(G)|^{\frac{1}{2}}$ and $\psi \in \lin(H)$ such that $\psi^G \in \nl(G)$. If $Z(G)$ is cyclic, then $H$ is abelian. 
\end{corollary}

\begin{corollary} \label{cor:reqpairVZp5}
	Suppose $G$ is a VZ $p$-group of order $\leq p^5$ ($p$ is any prime). Let $H$ be a subgroup $G$ of index $|G/Z(G)|^{\frac{1}{2}}$ with $\psi \in \lin(H)$ such that $\psi^G \in \nl(G)$. Then $H$ is abelian.
\end{corollary}
\begin{proof}
	Suppose $G$ is a VZ $p$-group of order $\leq p^5$. In this case, $\cd(G) \in \{ \{1, p\}, \{1, p^2 \} \}$.
	If $\cd(G) = \{1, p\}$, then by Corollary \ref{coro:VZreqpairabelian}, $H$ is abelian. If $\cd(G) = \{1, p^2\}$, then $\sqrt{|G/Z(G)|}=p^2$ and hence  $|G'| = |Z(G)| = p$. Again, by Corollary \ref{coro:VZreqpairabelian}, it follows that $H$ is abelian.
\end{proof}
\begin{remark}
\textnormal{The conclusion of Corollary \ref{cor:reqpairVZp5} need not holds for higher order VZ $p$-groups. For instance, consider the group
	\begin{align*}
		G= G_{(15,1)} = \langle & \alpha_1, \alpha_2, \alpha_3, \alpha_4, \alpha_5, \alpha_6 : [\alpha_3, \alpha_5]= \alpha_1, [\alpha_4, \alpha_5]=\alpha_2, [\alpha_3, \alpha_6]=\alpha_2, [\alpha_4, \alpha_6]=\alpha_1^\nu,\\
		 & \alpha_1^p=\alpha_2^p=\alpha_3^p=\alpha_4^p=\alpha_5^p=\alpha_6^p=1 \rangle,
	\end{align*}
of order $p^6$ ($p \geq 7$), where $\nu$ denotes the smallest positive integer which is a quadratic non-residue (\text{mod} $p$) (see \cite{NewmanO`Brien}). Here, $Z(G) = G' = \langle \alpha_1, \alpha_2 \rangle$. It is easy to observe that $G$ is a VZ $p$-group and $\cd (G)=\{1, p^2\}$. Set $H = \langle \alpha_1, \alpha_2, \alpha_3, \alpha_5 \rangle$. Then $H'=\langle \alpha_1 \rangle$ and $H/H' = \langle \alpha_2H', \alpha_3H', \alpha_5H' \rangle \cong C_p \times C_p \times C_p$. Define $\bar{\psi} \in \Irr(H/H')$ such that $\bar{\psi}(\alpha_2H')=\zeta_p$, $\bar{\psi}(\alpha_3H')=1$, and $\bar{\psi}(\alpha_5H')=1$. Now, define a character $\psi$ of $H$ by taking the lift of $\bar{\psi} \in \Irr(H/H')$. Then, $\psi \in \lin(H)$ and $\psi {\downarrow}_{Z(G)}\in \Irr(Z(G)| G{}')=\Irr(Z(G))\setminus 1_{Z(G)}$, where $1_{Z(G)}$ is the trivial character of $Z(G)$. 
Then by Proposition \ref{prop:VZreqpair}, $\psi^G \in \text{nl}(G)$, however $H$ is non-abelian.}
\end{remark}

\subsection{Rational group algebra of a VZ $p$-group.}\label{subsec:VZalgebra}
 Let $G$ be a VZ-group. Then the Wedderburn decomposition of $\mathbb{C}G$ is as follows:
 \[\mathbb{C}G \cong |G/G'|\mathbb{C}\bigoplus (|Z(G)| - |Z(G)/G'|) M_{|G/Z(G)|^{\frac{1}{2}}}(\mathbb{C}),\]
 where $M_{|G/Z(G)|^{\frac{1}{2}}}(\mathbb{C})$ denotes the ring of matrices of order $|G/Z(G)|^{\frac{1}{2}}$ over $\mathbb{C}$. In this subsection, we compute the Wedderburn decomposition of the rational group algebra of a finite VZ $p$-group.\\
%

Let $G$ be a finite group and let $\chi$, $\psi\in \Irr(G)$. It is well known that if $\chi$ and $\psi$ are Galois conjugates over $\mathbb{Q}$, then $\ker(\chi)=\ker(\psi)$. Now, we begin with the following easy observations. 
 \begin{lemma}\label{lem:galoislinear}
 	Let $G$ be a finite group, and let $\chi, \psi \in \lin(G)$ such that $\ker(\chi) = \ker(\psi)$. Then $\chi$ and $\psi$ are Galois conjugates over $\mathbb{Q}$.
 \end{lemma} 
In general, if $\chi$, $\psi\in \nl(G)$ such that $\ker(\chi)=\ker(\psi)$, then $\chi$ may not be Galois conjugate to $\psi$ over $\mathbb{Q}$. But in the case of VZ groups, this is true. 
\begin{lemma}\label{lem:galoisVZ}
	Let $G$  be a VZ-group and let $\chi, \psi \in \Irr(G)$. Then $\chi$ and $\psi$ are Galois conjugates over $\mathbb{Q}$ if and only if $\ker(\chi) = \ker(\psi)$.
\end{lemma} 
\begin{proof} If $\chi, \psi\in \lin(G)$, then the result follows from Lemma \ref{lem:galoislinear}. Now let $\chi, \psi\in \nl(G)$ such that $\ker(\chi)=\ker(\psi)$. Then in the view of \eqref{VZ}, there exist $\mu, \nu \in \Irr(Z(G)~|~G{}')$ such that $\chi \downarrow_{Z(G)}=\mu$ and $\psi \downarrow_{Z(G)}=\nu$. Observe that $\ker(\chi)=\ker(\mu)$ and $\ker(\psi)=\ker(\nu)$. Thus, $\mu$ and $\nu$ are Galois conjugates and hence $\chi$ and $\psi$ are Galois conjugates over $\mathbb{Q}$. This completes the proof. 
\end{proof}

\begin{lemma}\label{lemma:Ayoub}
	Consider a finite abelian group $G$, where $d$ divides the exponent of $G$. Let $a_d$ denote the number of cyclic subgroups of $G$ with order $d$. Then the number of non-Galois conjugate characters $\chi$ satisfying $\mathbb{Q}(\chi)=\mathbb{Q}(\zeta_d)$ is precisely $a_d$.
\end{lemma}
\begin{proof}
	See \cite[Lemma 1]{Ayoub}.
\end{proof}

Analogue to Lemma \ref{lemma:Ayoub}, we have following lemma for VZ-groups. 
\begin{lemma}\label{lemma:VZgaloisconjugates}
	Consider a VZ-group $G$. Let $\chi_\mu\in \nl(G)$ (defined in Equation \eqref{VZ}). Assume that $a_d$ and $a_d'$ represent the number of cyclic subgroups of order $d$ of $Z(G)$ and $Z(G)/G'$, respectively. Then the following statements hold:
	\begin{enumerate}
		\item The number of non-Galois conjugates non-linear characters $\chi_\mu$ of $G$ satisfying $\mathbb{Q}(\chi_\mu) = \mathbb{Q}(\zeta_d)$, where $d \mid \exp(Z(G))$ but $d \nmid \exp(Z(G)/G')$, is equal to $a_d$.
		\item The number of non-Galois Conjugates non-linear characters $\chi_\mu$ of $G$ satisfying $\mathbb{Q}(\chi_\mu) = \mathbb{Q}(\zeta_d)$, where $d \mid \exp(Z(G))$ and $d \mid \exp(Z(G)/G')$, is equal to $a_d - a_d'$.
	\end{enumerate}
\end{lemma}
\begin{proof}
	By the one-to-one correspondence between $\nl(G)$ and $\Irr(Z(G) | G')$, we get that $\ker(\chi_\mu) = \ker(\mu)$ and $\mathbb{Q}(\chi_\mu) = \mathbb{Q}(\mu)$.
Observe that $\Irr(Z(G))=\Irr(Z(G)|G{}') \sqcup \Irr(Z(G)/G{}')$. Hence if $d \mid \exp(Z(G))$ but $d\nmid \exp(Z(G)/G')$, then by Lemma \ref{lemma:Ayoub}, the number of non-Galois conjugate characters $\mu \in \Irr(Z(G)|G{}')$ such that $\mathbb{Q}(\mu) = \mathbb{Q}(\zeta_d)$ is equal to  $a_d$. This proves Lemma \ref{lemma:VZgaloisconjugates} (1). \\
	Similarly, if $d \mid \exp(Z(G))$ and $d \mid \exp(Z(G)/G')$, then the number of non-Galois conjugate characters $\mu\in \Irr(Z(G)|G{}')$ such that $\mathbb{Q}(\mu) = \mathbb{Q}(\zeta_d)$ is the difference between the number of non-Galois conjugate characters in $\Irr(Z(G))$ whose character field is $\mathbb{Q}(\zeta_d)$ and the number of non-Galois characters in $\Irr(Z(G)/G')$ whose character field is $\mathbb{Q}(\zeta_d)$. Hence by using Lemma \ref{lemma:Ayoub}, we get the Lemma
	\ref{lemma:VZgaloisconjugates} (2).
\end{proof}

Now, we  prove Theorem \ref{thm:wedderburn VZ}, which provides the Wedderburn decomposition of a VZ $p$-group, where $p$ is an odd prime. 
\begin{proof}[Proof of Theorem \ref{thm:wedderburn VZ}]
	Let $G$ be a finite VZ $p$-group (odd prime $p$) and $\chi \in \Irr(G)$. Suppose $\rho$ is an irreducible $\mathbb{Q}$-representation of $G$ which affords the character $\Omega(\chi)$. Let $A_\mathbb{Q}(\chi)$ be the simple component of the Wedderburn decomposition of $\mathbb{Q}G$ corresponding to $\rho$, which is isomorphic to $M_n(D)$ for some $n\in \mathbb{N}$ and a division ring $D$. From Lemma \ref{lemma:schurindexpgroup}, $m_\mathbb{Q}(\chi) = 1$ and from Lemma \ref{Reiner} we have $[D: Z(D)] = m_\mathbb{Q}(\chi)^2$ and $Z(D) = \mathbb{Q}(\chi)$. Therefore, $D = Z(D) = \mathbb{Q}(\chi)$. Now consider $\rho = {\rho_1} \oplus {\rho_2} \oplus \cdots \oplus {\rho_k}$, where for $1\leq i \leq k$, $\rho_i$ is a complex irreducible representation of $G$ affording $\chi^{\sigma_i}$ for some $\sigma_i \in \gal(\mathbb{Q}(\chi):\mathbb{Q})$. Here, $k=|\mathbb{Q}(\chi): \mathbb{Q}|$. Since $m_\mathbb{Q}(\chi) = 1$, we observe that $n = \chi(1)$.\\
	Let $\chi \in \lin(G)$ and suppose $\rho$ is the irreducible $\mathbb{Q}$-representation of $G$ affording $\Omega(\chi)$. Let $\bar{\chi} \in \Irr(G/G')$ such that $\bar{\chi}(gG') = \chi(g)$. Hence, $A_\mathbb{Q}(\bar{\chi}) \cong \mathbb{Q}(\bar{\chi})$. Since $G/G'$ is abelian, according to Lemma \ref{Perlis-walker}, the simple components of the Wedderburn decomposition of $\mathbb{Q}G$ corresponding to all irreducible $\mathbb{Q}$-representations of $G$ whose kernels contain $G'$ contribute
  \[\bigoplus_{d_1|m_1}a_{d_1}\mathbb{Q}(\zeta_{d_1})\]
  in $\mathbb{Q}G$, where $m_1$ is the exponent of $G/G'$ and $a_{d_1}$ is the number of cyclic subgroups of $G/G'$ of order $d_1$.\\
  Now, let $\rho$ be a irreducible $\mathbb{Q}$-representation of $G$ which affords the character $\Omega(\chi_{\mu})$, where $\chi_{\mu}\in \nl(G)$ as defined in \eqref{VZ}.  Here, $\chi_\mu(1) = |G/Z(G)|^{\frac{1}{2}}$ and $\mathbb{Q}(\chi_\mu) = \mathbb{Q}(\mu)$. Therefore by the above discussion, $A_\mathbb{Q}(\chi_\mu) \cong M_{|G/Z(G)|^{\frac{1}{2}}}(\mathbb{Q}(\mu))$. Observe that $\mathbb{Q}(\chi_\mu) = \mathbb{Q}(\mu) = \mathbb{Q}(\zeta_{d})$, for some $d \mid \exp(Z(G))$.  
Now, we have two cases:\\ 
{\bf  Case 1 ($d \mid \exp(Z(G))$ but $d \nmid \exp(Z(G)/G')$).} In this case, from Lemma \ref{lemma:VZgaloisconjugates} (1), the number of irreducible $\mathbb{Q}$ representations of $G$ which affords the character $\Omega(\chi_{\mu})$ is equal to the number of cyclic subgroups of $Z(G)$ of order $d$, where $\mathbb{Q}(\chi_{\mu})=\mathbb{Q}(\mu)=\mathbb{Q}(\zeta_d)$.\\
{\bf Case 2 ($d \mid \exp(Z(G))$ and $d \mid \exp(Z(G)/G')$).} In this case, from Lemma \ref{lemma:VZgaloisconjugates} (2), the number of irreducible $\mathbb{Q}$ representations of $G$ which affords the character $\Omega(\chi_{\mu})$ is equal to the difference of the number of cyclic subgroups of $Z(G)$ of order $d$ and the number of cyclic subgroups of $Z(G)/G{}'$ of order $d$, where $\mathbb{Q}(\chi_{\mu})=\mathbb{Q}(\mu)=\mathbb{Q}(\zeta_d)$.\\
Let $m_2$ and $m_3$ be the exponents of $Z(G)$ and $Z(G)/G'$ respectively. Then by the above discussion, the simple components of the Wedderburn decomposition of $\mathbb{Q}G$ corresponding to all irreducible $\mathbb{Q}$-representations of $G$ whose kernels do not contain $G'$ contribute
   \[\bigoplus_{d_2|m_2, d_2 \nmid m_3} a_{d_2}M_{|G/Z(G)|^{\frac{1}{2}}}(\mathbb{Q}(\zeta_{d_2}))\bigoplus_{d_2|m_2, d_2|m_3}(a_{d_2}-a_{d_2}')M_{|G/Z(G)|^{\frac{1}{2}}}(\mathbb{Q}(\zeta_{d_2}))\]
   in $\mathbb{Q}G$, where $a_{d_2}$ and $a_{d_2}'$ are the number of cyclic subgroups of $Z(G)$ of order $d_2$ and the number of cyclic subgroups of $Z(G)/G'$ of order $d_2$ respectively. Therefore, the result follows.
 \end{proof}

\begin{corollary}\label{cor:wedderburn VZ cyclic}
Let $G$ be a finite VZ $p$-group (odd prime $p$) with cyclic center $Z(G)$. Then the Wedderburn decomposition of the group algebra $\mathbb{Q}G$ is give by,
	\[\mathbb{Q}G \cong \mathbb{Q}(G/G') \bigoplus M_{|G/Z(G)|^{\frac{1}{2}}}(\mathbb{Q}(\zeta_{|Z(G)|})). \]
\end{corollary}
\begin{proof}
	Let $G$ be a VZ $p$-group (odd prime $p$). It is known that $G' \subseteq Z(G)$ and $G'$ is an elementary abelian $p$-group. Since $Z(G)$ is cyclic, $|G'|=p$. Let $\mu \in \Irr(Z(G) | G')$. It follows that $\mu$ is faithful and $\mathbb{Q}(\mu) = \mathbb{Q}(\zeta_{|Z(G)|})$. Therefore, all non-linear complex irreducible characters $G$ are faithful and Galois conjugates to each other. This completes the proof.
\end{proof}


\begin{corollary}\label{coro:extraspecialodd}
	Let $G$ be an extraspecial $p$-group (odd prime $p$) of order $p^{1+2n}$. Then
	\[\mathbb{Q}G \cong \mathbb{Q} \bigoplus (p^{2n-1}+p^{2n-2}+\dots+p+1) \mathbb{Q}(\zeta_p) \bigoplus M_{p^n}(\mathbb{Q}(\zeta_p)).\]
\end{corollary}
	\begin{proof} One can easily observe that $G$ is a VZ $p$-group and 
		$|G/Z(G)|^{\frac{1}{2}} = p^n=\chi(1)$, where $\chi\in \nl(G)$. Therefore, the result follows from Lemma \ref{Perlis-walker} and Corollary \ref{cor:wedderburn VZ cyclic}.
	\end{proof}
   \begin{remark}
   \textnormal{	
   \begin{enumerate}
   \item Corollary \ref{coro:extraspecialodd} shows that the rational group algebras of two non-isomorphic groups may be isomorphic. 
  \item Suppose $G$ is a non-abelian $p$-group of order $p^3$ (odd prime $p$). Then by Corollary \ref{coro:extraspecialodd}, $$\mathbb{Q}G \cong \mathbb{Q} \bigoplus (p+1) \mathbb{Q}(\zeta_p) \bigoplus M_{p}(\mathbb{Q}(\zeta_p)).$$ This is also computed in \cite[Theorem 3 and Theorem 4]{BM14}. The authors computed Wedderburn components of $\mathbb{Q}G$ by using Shoda pair concept.
  \end{enumerate}}
   \end{remark}
\begin{corollary} Let $G$ and $H$ are two isoclinic VZ $p$-groups (odd prime $p$) of same order. Then the Wedderburn decompositions of $\mathbb{Q}G$ and $\mathbb{Q}H$ are the same if and only if $G/G' \cong H/H'$ and $Z(G) \cong Z(H)$.
\end{corollary}
\begin{proof}
	The result follows from Theorem \ref{thm:wedderburn VZ}.
\end{proof}

Now, we discuss rational group algebra of a VZ $2$-group.
	\begin{lemma}\label{lemma:characterfielfVZ 2-gp}
		Let $G$ be a VZ $2$-group. Suppose $\chi \in \nl(G)$ such that $m_\mathbb{Q}(\chi) = 2$. Then $\mathbb{Q}(\chi) = \mathbb{Q}$.
	\end{lemma}
	\begin{proof} The result follows from Lemma \ref{lemma:existanceCQDSd} and Proposition \ref{prop:reqpairVZ2-gp}. 
	\end{proof}
Now we prove Theorem \ref{thm:WedderburnVZ 2-gp}, which provides a combinatorial description for the Wedderburn decomposition of rational group algebra of a VZ $2$-group.
	
	\begin{proof}[Proof of Theorem \ref{thm:WedderburnVZ 2-gp}]
		Let $\chi_{\mu} \in \nl(G)$ as defined in \eqref{VZ}. We have two cases for nonlinear irreducible complex characters of $G$. \\
{\bf Case 1 ($m_{\mathbb{Q}}(\chi_{\mu})=2, \chi_{\mu}\in \nl(G)$).} By Lemma \ref{lemma:characterfielfVZ 2-gp}, there   are $k$ many rational irreducible representations of $G$ which corresponds to $k$ many  simple component of $\mathbb{Q}G$, denoted as $A_\mathbb{Q}(\chi_{\mu})$. We know that $A_\mathbb{Q}(\chi_{\mu}) \cong M_n(D)$ for some $n\in \mathbb{N}$ and a division ring $D$. By \cite[Theorem 2.4]{Y1}, we have $n = \frac{1}{2}|G/Z(G)|^{\frac{1}{2}}$. Now by using Lemma \ref{Reiner} and Lemma \ref{lemma:characterfielfVZ 2-gp}, we get $Z(D) = \mathbb{Q}(\chi) = \mathbb{Q}$ and $[D : \mathbb{Q}] = 4$. It is a well-known fact that $\mathbb{H}(\mathbb{Q})$ is the only division ring satisfying $Z(\mathbb{H}(\mathbb{Q})) = \mathbb{Q}$ and $[\mathbb{H}(\mathbb{Q}) : \mathbb{Q}] = 4$. Therefore, we get $A_\mathbb{Q}(\chi_\mu) = M_{\frac{1}{2}|G/Z(G)|^{\frac{1}{2}}}(\mathbb{H}(\mathbb{Q}))$.\\
{\bf Case 2 ($m_{\mathbb{Q}}(\chi_{\mu})=1, \chi_{\mu}\in \nl(G)$).} Here, $A_{\mathbb{Q}}(\chi_{\mu})\cong M_n(D)$ for some $n\in \mathbb{N}$ and a division ring $D$. By \cite[Theorem 2.4]{Y1}, we have $n =|G/Z(G)|^{\frac{1}{2}}$. Now again by using Lemma \ref{Reiner}, we get $D = \mathbb{Q}(\chi_{\mu})$. Now we have two sub-cases:\\
{\bf Sub-case 2(1)($\mathbb{Q}(\chi_{\mu})=\mathbb{Q}$).} From the Case (1) and Lemma \ref{lemma:characterfielfVZ 2-gp}, there are $a_2-a_2{}'-k$ many Galois conjugacy classes of complex irreducible characters $\chi_{\mu}$ such that $m_{\mathbb{Q}}(\chi_\mu)=1$ and $\mathbb{Q}(\chi_{\mu})=\mathbb{Q}$. Therefore, $QG$ contains 
$(a_2-a_2'-k)M_{|G/Z(G)|^{\frac{1}{2}}}(\mathbb{Q})$. \\
{\bf Sub-case 2(2)($\mathbb{Q}(\chi_{\mu})\neq \mathbb{Q}$).} In this sub-case 
$\mathbb{Q}(\chi_{\mu})=\mathbb{Q}(\mu)=\mathbb{Q}(\zeta_d)$, where $d\geq 4$. Observe that either $d \mid \exp(Z(G))$ but $d\nmid \exp(Z(G)/G{}')$ or $d \mid \exp(Z(G))$ and $d\mid \exp(Z(G)/G{}')$. Now by using similar argument mentioned in the Proof of Theorem \ref{thm:wedderburn VZ}, $\mathbb{Q}G$ contains
\[\bigoplus_{d_2\mid m_2, d_2 \nmid m_3} a_{d_2}M_{|G/Z(G)|^{\frac{1}{2}}}(\mathbb{Q}(\zeta_{d_2})) \bigoplus_{d_2\mid m_2, d_2 \mid m_3}(a_{d_2}-a_{d_2}')M_{|G/Z(G)|^{\frac{1}{2}}}(\mathbb{Q}(\zeta_{d_2})).\] 
This completes the discussion of Case (2).\\
Now by using Lemma \ref{Perlis-walker}, Case (1) and Case (2), we get the result. 
	\end{proof}

	\begin{corollary}
		Let $G$ be a VZ $2$-group $G$ such that $Z(G)$ is cyclic and $|Z(G)| \geq 4$. Then the following hold:
		\begin{enumerate}
		\item $m_\mathbb{Q}(\chi) = 1$ for each $\chi \in \Irr(G)$.
		\item $\mathbb{Q}G \cong \mathbb{Q}(G/G') \bigoplus M_{|G/Z(G)|^{\frac{1}{2}}}(\mathbb{Q}(\zeta_{|Z(G)|})).$
\end{enumerate}		 
	\end{corollary}
	\begin{proof}
		It follows from Theorem \ref{thm:WedderburnVZ 2-gp} and Lemma \ref{lemma:characterfielfVZ 2-gp}.
	\end{proof}
	\begin{corollary}
		Suppose $G$ is a VZ $2$-group with elementary abelian center. Then 
		\[\mathbb{Q}G \cong \mathbb{Q}(G/G') \bigoplus kM_{\frac{1}{2}|G/Z(G)|^{\frac{1}{2}}}(\mathbb{H}(\mathbb{Q})) \bigoplus k' M_{|G/Z(G)|^{\frac{1}{2}}}(\mathbb{Q}),\]
		where $k$ and $k'$ denote the number of non-linear complex irreducible characters of $G$ with Schur index $2$ and the number of non-linear complex irreducible characters of $G$ with Schur index of $1$, respectively. 
\end{corollary}

The counting of rational irreducible representations of an abelian group can be determined using Lemma \ref{Perlis-walker} and Remark \ref{Rational Abelian1}. Corollary \ref{cor:VZcounting} provides a characterization for counting irreducible rational representations of VZ-groups.
\begin{corollary}\label{cor:VZcounting}
	For a VZ-group $G$, let $n$, $x$, $y$, and $z$ represent the total number of irreducible rational representations of $G$, $G/G'$, $Z(G)$, and $Z(G)/G'$, respectively. It follows that $n = x + y - z$.
\end{corollary}
\begin{proof} Observe that $|\{ \eta \in \Irr_{\mathbb{Q}}(G) : G' \subseteq \ker(\eta) \}| = x$. Further,
	\begin{align*}
		|\{ \eta \in \Irr_{\mathbb{Q}}(G) : G' \nsubseteq \ker(\eta) \}| =  &\text{ the number of Galois conjugacy classes of } \Irr(Z(G)) \, \, \text{over} \,\, \mathbb{Q}\\
		&- \text{the number of Galois conjugacy classes of } \Irr(Z(G)/G') \, \, \text{over} \, \, \mathbb{Q}.
	\end{align*}
	This completes the proof.
\end{proof}
In the case of cyclic center, we have the following corollary. 
\begin{corollary}
	If $G$ is a VZ $p$-group and $Z(G)$ is cyclic, then $G$ has only one rational irreducible representation whose kernel does not contain $G'$.
\end{corollary}

\subsection{Primitive central idempotents in rational group algebras of VZ-groups.} \label{subsec:pci} Let $G$ be a finite group. An element $e$ in $\mathbb{Q}G$ is an idempotent if $e^2=e$. A primitive central idempotent $e$ in $\mathbb{Q}G$ is one that belongs to the center of $\mathbb{Q}G$ and cannot be expressed as $e=e'+e''$, where $e'$ and $e''$ are non-zero idempotents such that $e'e''=0$. It is well-known that a complete set of primitive central idempotents of $\mathbb{Q}G$ determines the decomposition of $\mathbb{Q}G$ into a direct sum of simple sub-algebras. Specifically, if $e$ is a primitive central idempotent of $\mathbb{Q}G$, then the corresponding simple component of $\mathbb{Q}G$ is $\mathbb{Q}Ge$. For $\chi \in \Irr(G)$, the expression
\[e(\chi)=\frac{\chi(1)}{|G|}\sum_{g\in G}\chi(g)g^{-1}\]
defines a primitive central idempotent of $\mathbb{C}G$. In fact, the set $\{e(\chi) : \chi \in \Irr(G)\}$ forms a complete set of primitive central idempotents of $\mathbb{C}G$. Moreover, for $\chi \in \Irr(G)$, we define
\[e_{\mathbb{Q}}(\chi) := \sum_{\sigma \in \gal(\mathbb{Q}(\chi) / \mathbb{Q})} e(\chi^{\sigma}).\]
Then $e_{\mathbb{Q}}(\chi)$ is a primitive central idempotent in $\mathbb{Q}G$.\\
For a subset $X$ of $G$, define
\[\widehat{X}=\frac{1}{|X|}\sum_{x\in X}x \in \mathbb{Q}G,\]
and for a normal subgroup $N$ of $G$, define
\[\epsilon(G, N)= \begin{cases}
	\widehat{G} &  \text{if} \quad G= N;\\
	\prod_{D/N \in M(G/N)}(\widehat{N}-\widehat{D}), & \text{otherwise},
\end{cases}\]
where $M(G/N)$ represents the set of minimal non-trivial normal subgroups $D/N$ of $G/N$, with $D$ being a subgroup of $G$ that contains $N$.
In this subsection, we compute a complete set of primitive central idempotents of the rational group algebra of a VZ-group. Let us start with a general result. 

\begin{lemma}\cite[Lemma 3.3.2]{JR}\label{lemma:pcilin}
	Let $G$ be a finite group. If $\chi \in \lin(G)$ and $N=\ker(\chi)$, then the following hold:
	\begin{enumerate}
	\item $e_{\mathbb{Q}}(\chi)= \epsilon(G, N)$.
	\item $\mathbb{Q}G\epsilon(G, N) \cong \mathbb{Q}(\zeta_{|G/N|})$. 
	\end{enumerate}
\end{lemma}
 
 Theorem \ref{thm:pciVZ} provides a characterization of primitive central idempotents and their corresponding simple components in $\mathbb{Q}G$ for a VZ $p$-group.
 \begin{theorem}\label{thm:pciVZ}
 	Let $G$ be a VZ-group and let $\chi_\mu \in \nl(G)$ (as defined in Equation \eqref{VZ}) with $N=\ker(\chi_\mu)=\ker(\mu)$. Then the following statements hold:
 	\begin{enumerate}
 		\item $e_{\mathbb{Q}}(\chi_\mu)= \epsilon(Z(G), N)$. 
 		\item If $G$ is a VZ $p$-group (odd prime $p$), then $\mathbb{Q}G\epsilon(G, N) \cong M_{|G/Z(G)|^{\frac{1}{2}}}(\mathbb{Q}(\zeta_{|Z(G)/N|}))$.
 \item If $G$ is a VZ $2$-group, then $\mathbb{Q}G\epsilon(G, N) \cong M_{|G/Z(G)|^{\frac{1}{2}}}(\mathbb{Q}(\zeta_{|Z(G)/N|}))$ when $m_\mathbb{Q}(\chi_{\mu})=1$, and $\mathbb{Q}G\epsilon(G, N) \cong M_{\frac{1}{2}|G/Z(G)|^{\frac{1}{2}}}(\mathbb{H}(\mathbb{Q}))$ when $m_\mathbb{Q}(\chi_{\mu})=2$.		
 	\end{enumerate} 
 \end{theorem}
\begin{proof} From \eqref{VZ}, it is easy to observe that $e(\chi_\mu)=e(\mu)$.
Furthermore, we can observe:
   \begin{align*}
   	e_{\mathbb{Q}}(\chi_\mu) = & \sum_{\sigma \in \gal(\mathbb{Q}(\chi_\mu) / \mathbb{Q})} e(\chi_\mu^{\sigma})\\
   	= & \sum_{\sigma \in \gal(\mathbb{Q}(\mu) / \mathbb{Q})} e(\chi_{\mu^{\sigma}})\\
   	= & \sum_{\sigma \in \gal(\mathbb{Q}(\mu) / \mathbb{Q})} e({\mu^{\sigma}})\\
   	= & e_{\mathbb{Q}}(\mu)\\
   	= & \epsilon(Z(G), N) \quad (\text{from Lemma \ref{lemma:pcilin}}),
   \end{align*}
where $N=\ker(\mu)$. Moreover, if $G$ is a VZ $p$-group (odd prime $p$), from Theorem \ref{thm:wedderburn VZ}, the simple component of $\mathbb{Q}G$ corresponding to $\chi_\mu \in \nl(G)$ is given by: 
\[A_{\mathbb{Q}}(\chi_\mu)\cong M_{|G/Z(G)|^{\frac{1}{2}}}(\mathbb{Q}(\mu))=M_{|G/Z(G)|^{\frac{1}{2}}}(\mathbb{Q}(\zeta_{|Z(G)/N|})).\]
Hence, we have:
\[\mathbb{Q}G e_{\mathbb{Q}}(\chi_\mu) = \mathbb{Q}G\epsilon(G, N) \cong M_{|G/Z(G)|^{\frac{1}{2}}}(\mathbb{Q}(\zeta_{|Z(G)/N|})).\]
Similarly, (3) follows from Theorem \ref{thm:WedderburnVZ 2-gp}.
  \end{proof}
     \begin{remark}\label{rem:VZPCI}\textnormal{In \cite[Corollary 2]{BP}, primitive central idempotents of the rational group algebra of VZ-groups have been computed. However, our approach, which is based on character properties, offers a direct proof.}
    \end{remark}
\section{$p$-group of order $p^3$} \label{sec:p^3}
%
Let $G$ be a non-abelian $p$-group (odd prime $p$) of order $p^3$. It is well known that $G$ will be isomorphic to one of the following two groups:
\begin{align*}
	\Phi_2(21) =& \langle \alpha, \alpha_1, \alpha_2 : [\alpha_1, \alpha]=\alpha^p=\alpha_2, \alpha_1^p=\alpha_2^p=1 \rangle, \text{ and}\\
	\Phi_2(111) =& \langle \alpha, \alpha_1, \alpha_2 : [\alpha_1, \alpha]=\alpha_2, \alpha^p=\alpha_1^p=\alpha_2^p=1 \rangle,
\end{align*}
(see \cite[Subsection 4.3]{RJ}). 
It is easy to check that both the groups are VZ $p$-groups.
In both cases, we have $Z(G) = G' = \langle \alpha_2 \rangle \cong C_p$, $cd(G)=\{1,p\}$, $|\nl(G)|=|Z(G)|-1$ and $G/G' = \langle \alpha G', \alpha_1 G' \rangle \cong C_p \times C_p$. Since $G$ is a VZ $p$-group,
the non-linear characters of $G$ can be defined as follows:
 \begin{equation}\label{p^3}
 	\chi_\mu(g) = \begin{cases}
 		p\mu(g)  &\quad \text{ if } g \in Z(G),\\
 		0            &\quad \text{ otherwise }
 	\end{cases}
 \end{equation}
 where $\mu \in \Irr(Z(G) | G')$. The rational representations of $G$ are characterized in Proposition \ref{prop:rationalp^3}.
\begin{proposition}\label{prop:rationalp^3}
	Let $G$ be a non-abelian group of order $p^3$ (odd prime $p$). Then the following statements hold:
	\begin{enumerate}
		\item $|\Irr_{\mathbb{Q}}^{(1)}(G)|=1$, $|\Irr_{\mathbb{Q}}^{(\phi(p))}(G)|=p+1$, and $|\Irr_{\mathbb{Q}}^{(\phi(p^2))}(G)|=1$.
		\item A special required pair $(H, \psi_\mu)$ to determine a rational matrix representation of $G$ whose character is $\Omega(\chi_\mu)$ ($\chi_{\mu}$ is defined in \eqref{p^3}) is given by:
		$H = \langle \alpha_1, \alpha_2 \rangle$, and $\psi_\mu \in \lin(H)$ defined by $\psi_\mu(\alpha_1) = 1$ and $\psi_\mu(\alpha_2) = \mu(\alpha_2)$.
	\end{enumerate}
\end{proposition}
\begin{proof}
Proof of Proposition \ref{prop:rationalp^3} (1) is  obvious. \\
As $Z(G) \subset H$ and $\psi_\mu \downarrow_{Z(G)} = \mu$, we get $\psi_\mu^G = \chi_\mu$ (from Proposition \ref{prop:VZreqpair}). Furthermore, since $\psi_\mu(\alpha_1) = 1$ and $\psi_\mu(\alpha_2) = \mu(\alpha_2)$, it follows that $\mathbb{Q}(\psi_\mu) = \mathbb{Q}(\mu) = \mathbb{Q}(\chi_\mu)$. This completes the Proof of Proposition \ref{prop:rationalp^3} (2).
\end{proof}

\begin{remark}
	\textnormal{ In general, a required pair to find an irreducible rational matrix representation of $G$ whose character is $\Omega(\chi)$ may not be unique. For example, consider $G = \Phi_2(111)$ and $\chi_\mu \in \nl(G)$ as defined in \eqref{p^3}. Take $H_{1} = \langle \alpha, \alpha_2 \rangle$ and choose $\psi'_\mu \in \lin(H_{1})$ such that $\psi'_\mu(\alpha) = 1$ and $\psi'_\mu(\alpha_2) = \mu(\alpha_2)$.
	 The pair $(H_{1}, \psi'_\mu)$ is also a special required pair to find an irreducible rational matrix representation of $G$ which affords the character $\Omega(\chi_\mu)$.
}
\end{remark}

In Example \ref{ex:reqpairmatrix}, we provide an illustration to find out an irreducible rational matrix representation associated with a required pair.
\begin{example} \label{ex:reqpairmatrix}
	\textnormal{Consider 
		\[  G= \Phi_2(21) = \langle \alpha, \alpha_1, \alpha_2 : [\alpha_1, \alpha]=\alpha^p=\alpha_2, \alpha_1^p=\alpha_2^p=1 \rangle. \]
		We have $Z(G) = G' = \langle \alpha_2 \rangle \cong C_p$. Let $\mu \in \Irr(Z(G) | G')$ such that $\mu(\alpha_2) = \zeta_p$. The character $\chi_\mu$ defined in Equation \eqref{p^3} is a non-linear irreducible complex character of $G$.  From Proposition \ref{prop:rationalp^3}, a special required pair to find an irreducible rational matrix representation of $G$ affording the character $\Omega(\chi_\mu)$ is $(H, \psi_\mu)$, where $H = \langle \alpha_1, \alpha_2 \rangle$ and $\psi_\mu \in \lin(H)$ such that $\psi_\mu(\alpha_1) = 1$ and $\psi_\mu(\alpha_2) = \mu(\alpha_2)=\zeta_p$. Now, let $\Psi_\mu$ denote an irreducible rational matrix representation of degree $(p-1)$ of $H$ corresponding to $\Omega(\psi_\mu)$. The explicit form of $\Psi_\mu$ is given by:
	\[\Psi_\mu(\alpha_1)= \left(\begin{array}{ccccc}
		1 & 0 & 0 & \cdots & 0\\
		0 & 1 & 0 & \cdots & 0\\
		\vdots & \vdots & \vdots & \ddots \\
		0 & 0 & \cdots & 1 & 0\\
		0 & 0 & \cdots & \cdots & 1
	\end{array}\right) = I, \text{ and }
~ \Psi_\mu(\alpha_2)=\left(\begin{array}{ccccc}
	0 & 1 & 0 & \cdots & 0\\
	0 & 0 & 1 & \cdots & 0\\
	\vdots & \vdots & \vdots & \ddots \\
	0 & 0 & \cdots & 0 & 1\\
	-1 & -1 & \cdots & \cdots & -1
\end{array}\right),\]
(see Lemma \ref{lemma:YamadaLinear}). 
Set $\Psi_\mu(\alpha_2) = P$, where $P$ denotes a matrix of order $(p-1)$, and let $O$ denote the zero matrix of the same order. Then an irreducible rational matrix representation $\Psi^G$ of degree $p^2-p$ of $G$ affording the character $\Omega(\chi_\mu)$ is given by:
\[\Psi_\mu^G(\alpha)=\left(\begin{array}{ccccc}
	O & O & O & \cdots & P\\
	I & O & O & \cdots & O\\
	O & I & O & \cdots & O\\
	\vdots & \vdots & \vdots & \ddots \\
	O & O & \cdots & I & 0
\end{array}\right), \text{ and }
~ \Psi_\mu^G(\alpha_1)=\left(\begin{array}{ccccc}
	I & O & O & \cdots & O\\
	O & P & O & \cdots & O\\
	O & O & P^2 & \cdots & O\\
	\vdots & \vdots & \vdots & \ddots \\
	O & O & \cdots & \cdots & P^{p-1}
\end{array}\right).\]
$\Psi^G$ is a rational matrix representation of $G$ whose kernel does not contains $G{}'$. Now, use Lemma \ref{lemma:YamadaLinear}, to compute all irreducible rational matrix representations of $G$ whose kernel contains $G{}'$. This gives the complete description of all irreducible rational matrix representation of $G$.}
\end{example}
%
\begin{remark}\label{rem:repof8order}\textnormal{ If $G$ is a non-abelian group of order $8$, then either $G\cong Q_8$, or $G\cong D_8$. A special required pair to obtain an irreducible rational matrix representations of $G$ whose kernel does not contain $G{}'$ is determined in Subsection \ref{subsec:rrVZpgroups}.}
\end{remark}

\section{$p$-group of order $p^4$} \label{sec:p^4}
   In this section, we provide a comprehensive description of all inequivalent irreducible rational matrix representations and the Wedderburn decompositions of the rational group rings for all non-abelian groups of order $p^4$. It is easy to observe that if $G$ is a non-abelian group of order $p^4$, then $|Z(G)|=p~\textnormal{or}~p^2$, and $\cd(G)=\{1, p\}$. Moreover, if  $|Z(G)|=p$, then $G/Z(G)$ is non-abelian. As we know that a group $G$ is VZ-group if and only if  $\cd(G) = \{1, |G/Z(G)|^{\frac{1}{2}}\}$. Hence, a $p$-group of order $p^4$ of nilpotency class $2$ is a VZ $p$-group.  In the subsequent subsections, we will separately discuss the cases of groups of order $p^4$ of nilpotency class $2$ and nilpotency class $3$.

   
   \subsection{$p$-groups of order $p^4$ of nilpotency class $2$.}\label{subsec:p4class2}
   Let $G$ be a $p$-group of order $p^4$ of nilpotency class $2$. Then $|G'| = p$, $|Z(G)|=p^2$ and $\cd(G)=\{1,p\}$. Furthermore, 
   $|\lin(G)|=p^3$ and $|\nl(G)|=p^2-p$. Note that $G$ is a VZ $p$-group. Then by 
   \eqref{VZ}, a non-linear irreducible complex character is of the form $\chi_{\mu}$ and is given by: 
   \begin{equation}\label{VZp^4}
   	\chi_\mu(g) = \begin{cases}
   		p\mu(g)  &\quad \text{ if } g \in Z(G),\\
   		0            &\quad \text{ otherwise }
   	\end{cases}
   \end{equation}
   where $\mu \in \text{Irr}(Z(G) | G')$.\\
   
 For an odd prime $p$, we rely on James' classification of $p$-groups of order $p^4$ (see \cite{RJ}). There exist two distinct isoclinic families of non-abelian groups of order $p^4$, denoted as $\Phi_2$ and $\Phi_3$ (see \cite[Subsection 4.4]{RJ}).  Prajapati et al. \cite[Proposition 4.2]{SKP} have proved that all the groups belongs to isoclinic family $\Phi_2$ are VZ-groups. Note that all the groups belongs to isoclinic family $\Phi_3$ are non-VZ-groups. Theorem \ref{thm:reqpairVZp^4odd} provides a description of all inequivalent irreducible rational matrix representations for all VZ  $p$-groups of order $p^4$, where $p$ is an odd prime. 
 \begin{theorem}\label{thm:reqpairVZp^4odd}
 	Let $G$ be a non-abelian $p$-group (odd prime $p$) of order $p^4$ in the isoclinic family $\Phi_2$. Then Table \ref{t:1} determines all inequivalent irreducible rational matrix representations of $G$ whose kernels do not contain $G'$. 

\begin{tiny}
	\begin{longtable}[c]{|l|c|c|c|l|}
  \caption{Special required pair $(H, \psi_\mu)$ to obtain an irreducible rational matrix representation of $G \in \Phi_{2}$ which affords the character $\Omega(\chi_{\mu})$, where $\chi_{\mu}\in \nl(G)$ (defined in \eqref{VZp^4}) \label{t:1}}\\
		\hline
		\textnormal{Group} $G$ & $Z(G)$ & $G'$ & $H$ & $\psi_\mu\in \Irr(H)$ \textnormal{and} $\mu\in \Irr(Z(G)|G{}')$  \\
		\hline
		\endfirsthead
		\hline
		\multicolumn{5}{|c|}{Continuation of Table 1}\\
		\hline
		Group $G$ & $Z(G)$ & $G'$ & $H$ & $\psi$ \\
		\hline
		\endhead
		\hline
		\endfoot
		\hline
		\endlastfoot
		\hline \vtop{\hbox{\strut $\Phi_2(211)a=\langle \alpha, \alpha_1, \alpha_2, \alpha_3 : [\alpha_1, \alpha]=\alpha^p=\alpha_2,$}\hbox{\strut \qquad \qquad \qquad $\alpha_1^p=\alpha_2^p=\alpha_3^p=1\rangle$}}  & $\langle \alpha^p, \alpha_3 \rangle$ & $\langle \alpha^p \rangle$ & $\langle \alpha^p, \alpha_1, \alpha_3 \rangle$ & $\psi_\mu(h) = \begin{cases}
			\mu(\alpha^p)  &\quad \text{ if } h=\alpha^p,\\
			1  &\quad \text{ if } h=\alpha_1,\\
			\mu(\alpha_3)  &\quad \text{ if } h=\alpha_3\\
		\end{cases}$\\
		
		\hline \vtop{\hbox{\strut $\Phi_2(1^4)=\langle \alpha, \alpha_1, \alpha_2, \alpha_3 : [\alpha_1, \alpha]=\alpha_2,$}\hbox{\strut \qquad \qquad \quad $\alpha^p=\alpha_1^p=\alpha_2^p=\alpha_3^p=1\rangle$}} & $\langle \alpha_2, \alpha_3 \rangle$ & $\langle \alpha_2 \rangle$ & $\langle \alpha, \alpha_2, \alpha_3 \rangle$ & $\psi_\mu(h) = \begin{cases}
			1  &\quad \text{ if } h=\alpha,\\
			\mu(\alpha_2)  &\quad \text{ if } h=\alpha_2,\\
			\mu(\alpha_3)  &\quad \text{ if } h=\alpha_3\\
		\end{cases}$\\
	
    	\hline \vtop{\hbox{\strut $\Phi_2(31)=\langle \alpha, \alpha_1, \alpha_2 : [\alpha_1, \alpha]=\alpha^{p^2}=\alpha_2,$}\hbox{\strut \qquad \qquad \quad $\alpha_1^p=\alpha_2^p=1\rangle$}} & $\langle \alpha^p \rangle$ & $\langle \alpha^{p^2} \rangle$ & $\langle \alpha^p, \alpha_1\rangle$ & $\psi_\mu(h) = \begin{cases}
		\mu(\alpha^p)  &\quad \text{ if } h=\alpha^p,\\
		1  &\quad \text{ if } h=\alpha_1\\
	\end{cases}$\\
	
    	\hline \vtop{\hbox{\strut $\Phi_2(22)=\langle \alpha, \alpha_1, \alpha_2 : [\alpha_1, \alpha]=\alpha^p=\alpha_2,$}\hbox{\strut  \qquad \qquad \quad  $\alpha_1^{p^2}=\alpha_2^p=1\rangle$}} & $\langle \alpha^p, \alpha_1^p \rangle$ & $\langle \alpha^p \rangle$ & \vtop{\hbox{\strut$\langle \alpha^{-i}\alpha_1, \alpha^p \rangle$} \hbox{\strut $(0\leq i\leq p-1)$} }& $\psi_\mu(h) = \begin{cases}
        1  &\quad \text{ if } h=\alpha^{-i}\alpha_1,\\
		\mu(\alpha^p)  &\quad \text{ if } h=\alpha^p\\
	\end{cases}$\\

		\hline \vtop{\hbox{\strut $\Phi_2(211)b=\langle \alpha, \alpha_1, \alpha_2, \gamma : [\alpha_1, \alpha]=\gamma^p=\alpha_2,$}\hbox{\strut \qquad \qquad \qquad $\alpha^p=\alpha_1^p=\alpha_2^p=1\rangle$}} & $\langle \gamma\rangle$ & $\langle \gamma^p \rangle$ & $\langle \alpha_1, \gamma \rangle$ & $\psi_\mu(h) = \begin{cases}
			1  &\quad \text{ if } h=\alpha_1,\\
			\mu(\gamma)  &\quad \text{ if } h=\gamma\\
		\end{cases}$\\
		
		\hline \vtop{\hbox{\strut $\Phi_2(211)c=\langle \alpha, \alpha_1, \alpha_2 : [\alpha_1, \alpha]=\alpha_2,$}\hbox{\strut \qquad \qquad \qquad $\alpha^{p^2}=\alpha_1^p=\alpha_2^p=1\rangle$}} & $\langle \alpha^p, \alpha_2\rangle$ & $\langle \alpha_2 \rangle$ & $\langle \alpha^p, \alpha_1, \alpha_2 \rangle$ & $\psi_\mu(h) = \begin{cases}
			\mu(\alpha^p)  &\quad \text{ if } h=\alpha^p,\\
			1  &\quad \text{ if } h=\alpha_1,\\
			\mu(\alpha_2)  &\quad \text{ if } h=\alpha_2\\
		\end{cases}$\\
		\hline
		
	\end{longtable}
\end{tiny}
\end{theorem}

\begin{proof}
	Let $G \in \Phi_2$, and let $\chi_\mu \in \nl(G)$ (as defined in \eqref{VZp^4}). Suppose $(H, \psi_\mu)$ is a special required pair to obtain an irreducible rational matrix representation of $G$ which affords the character $\Omega(\chi_\mu)$. By Proposition \ref{prop:VZreqpair}, it follows that $Z(G) \subset H$ and $\psi_\mu\downarrow_{Z(G)} = \mu$. Further, from Corollary \ref{cor:reqpairVZp5}, $H$ is abelian. Since $(H, \psi_\mu)$ is a special required pair, $\mathbb{Q}(\psi_\mu) = \mathbb{Q}(\chi_\mu)$ and hence by Lemma \ref{lemma:fieldofcharacters}, we must choose $\psi_\mu \in \lin(H)$ such that $|\ker(\psi_\mu)| = |G/Z(G)|^{\frac{1}{2}} |\ker(\mu)| = p|\ker(\mu)|$. Consider $G=\Phi_2(22)$. Now, for each $0\leq i\leq (p-1)$, take 
	$H=\langle \alpha^{-i}\alpha_1, \alpha^p \rangle$ and define $\mu\in \Irr(Z(G)|G{}')$  as follows:
\[\mu(z) = \begin{cases}
	\zeta_p  &\quad \text{ if } z=\alpha^p,\\
	\zeta_p^i  &\quad \text{ if } z=\alpha_1^p\\
\end{cases}\]
where  $z\in Z(G)$. 
Observe that $(\alpha^{-i}\alpha_1)^p = \alpha^{-ip}\alpha_1^p$. Then $\psi_\mu \in \lin(H)$ (given in Table \ref{t:1}) satisfies $\psi_\mu(\alpha_1^p)= \left( \psi_\mu(\alpha^{-i}\alpha_1) \right)^p\left(\psi_\mu(\alpha^p)\right)^i = \left(\psi_\mu(\alpha^{p})\right)^{i}=\mu(\alpha_1^p)$. It is easy to check that a pair $(H, \psi_{\mu})$ satisfies the criteria of a special required pair. \\
It is routine to check that all the pairs $(H, \psi_{\mu})$ for the rest of the groups mentioned in Table \ref{t:1} also satisfy the criteria to being special required pairs.   
This shows that Table \ref{t:1} presents special required pairs $(H, \psi_\mu)$ to find all inequivalent irreducible rational matrix representations of $G$ whose kernels do not contain $G{}'$, where $G\in \Phi_2$.
This completes the proof of Theorem \ref{thm:reqpairVZp^4odd}.
\end{proof}

Proposition \ref{prop:countingVZp^4} provides the counting of rational irreducible representations of different degrees for all groups of order $p^4$ in $\Phi_2$.
\begin{proposition}\label{prop:countingVZp^4}
	Let $G$ be a non-abelian group of order $p^4$ (odd prime $p$) in $\Phi_2$. 
	\begin{enumerate}
		\item If $G=\Phi_2(211)a$, or $G=\Phi_2(1^4)$, then $|\Irr_{\mathbb{Q}}^{(1)}(G)|=1$, $|\Irr_{\mathbb{Q}}^{(\phi(p))}(G)|=p^2+p+1$, and $|\Irr_{\mathbb{Q}}^{(\phi(p^2))}(G)|=p$.
		\item If $G=\Phi_2(31)$, then $|\Irr_{\mathbb{Q}}^{(1)}(G)|=1$, $|\Irr_{\mathbb{Q}}^{(\phi(p))}(G)|=p+1$, $|\Irr_{\mathbb{Q}}^{(\phi(p^2))}(G)|=p$, and $|\Irr_{\mathbb{Q}}^{(\phi(p^3))}(G)|=1$.
		\item  If $G=\Phi_2(22)$, then $|\Irr_{\mathbb{Q}}^{(1)}(G)|=1$, $|\Irr_{\mathbb{Q}}^{(\phi(p))}(G)|=p+1$, and $|\Irr_{\mathbb{Q}}^{(\phi(p^2))}(G)|=2p$.
		\item  If $G=\Phi_2(211)b$, then $|\Irr_{\mathbb{Q}}^{(1)}(G)|=1$, $|\Irr_{\mathbb{Q}}^{(\phi(p))}(G)|=p^2+p+1$, and $|\Irr_{\mathbb{Q}}^{(\phi(p^3))}(G)|=1$.
		\item  If $G=\Phi_2(211)c$, then $|\Irr_{\mathbb{Q}}^{(1)}(G)|=1$, $|\Irr_{\mathbb{Q}}^{(\phi(p))}(G)|=p+1$, and $|\Irr_{\mathbb{Q}}^{(\phi(p^2))}(G)|=p+1$.
		\end{enumerate}
\end{proposition}
\begin{proof}
	Suppose $\chi \in \Irr(G)$ and $E(\chi)$ denotes the Galois conjugacy class of $\chi$ over $\mathbb{Q}$. Then the degree of the rational representation affording the character $\Omega(\chi)$ is $|E(\chi)|\chi(1)$. 
	\begin{itemize}
		\item[(1)] For $G=\Phi_2(211)a$, we have $Z(G)=\langle \alpha^p, \alpha_3 \rangle \cong C_p \times C_p$, $G'=\langle \alpha^p \rangle \cong C_p$, and $G/G'=\langle \alpha G', \alpha_1 G', \alpha_2 G' \rangle \cong C_p \times C_p \times C_p$. Observe that in $\Irr^{(1)}(G)$, there is a single Galois conjugacy class over $\mathbb{Q}$ with size $1$ and $\frac{p^3-1}{\phi(p)} = p^2+p+1$ many distinct Galois conjugacy classes over $\mathbb{Q}$ with size $\phi(p)$. If $\mu \in \Irr(Z(G) | G')$, then $[\mathbb{Q}(\mu) : \mathbb{Q}]=\phi(p)$. Hence, there are $\frac{p^2-p}{\phi(p)}=p$ distinct Galois conjugacy classes over $\mathbb{Q}$ with size $\phi(p)$ in $\Irr^{(p)}(G)$. Similar statements hold for $G=\Phi_2(1^4)$. This proves part (1).
		
		\item[(2)] For $G=\Phi_2(31)$, we have
		 $Z(G)=\langle \alpha^p \rangle \cong C_{p^2}$, $G'=\langle \alpha^{p^2} \rangle \cong C_p$, and $G/G'=\langle \alpha G', \alpha_1 G' \rangle \cong C_{p^2} \times C_p$. In $\Irr^{(1)}(G)$, there is one Galois conjugacy class over $\mathbb{Q}$ with size $1$, $\frac{1\times \phi(p)+\phi(p)\times1+\phi(p)\times\phi(p)}{\phi(p)} = p+1$ distinct Galois conjugacy classes over $\mathbb{Q}$ with size $\phi(p)$, and $\frac{\phi(p^2)\times p}{\phi(p^2)}=p$ distinct Galois conjugacy classes over $\mathbb{Q}$ with size $\phi(p^2)$. If $\mu \in \Irr(Z(G) | G')$, then $[\mathbb{Q}(\mu) : \mathbb{Q}]=\phi(p^2)$. Consequently, there is only one ( $\frac{p^2-p}{\phi(p^2)}=1$) Galois conjugacy class over $\mathbb{Q}$ with size $\phi(p^2)$ in $\Irr^{(p)}(G)$. This completes the proof of part (2).
	\end{itemize}
By using similar arguments, we get the proof of the remaining parts of Proposition \ref{prop:countingVZp^4}.
\end{proof}

   \begin{remark}
   	\textnormal{Let $G$ and $H$ be isoclinic groups with the same order. According to Lemma \ref{lemma:isoclinicCharacter}, we have $|\Irr^{(k)}(G)| = |\Irr^{(k)}(H)|$. However, it is important to note that $|\Irr_{\mathbb{Q}}^{(k)}(G)|$ may not be equal to $|\Irr_{\mathbb{Q}}^{(k)}(H)|$ (see Proposition \ref{prop:countingVZp^4}).}
   \end{remark}

Now, let $G$ be a non-abelian 2-group of order $16$ of nilpotency class $2$. Then $|\nl(G)|=2$. We take presentations of $2$-groups from Burnside's Book \cite{Burnside}. Theorem \ref{thm:reqpairVZ16} provides a description of all inequivalent irreducible rational matrix representations of $G$ whose kernels do not contain $G'$.

\begin{theorem}\label{thm:reqpairVZ16}
	Let $G$ be a non-abelian $2$-group of order $16$ of nilpotency class $2$. Then Table \ref{t:2} determines all inequivalent irreducible rational matrix representations of $G$ whose kernels do not contain $G'$.
\begin{tiny}
	\begin{longtable}[c]{|l|c|c|c|l|}
		\caption{Special required pair $(H, \psi_\mu)$ to obtain an irreducible rational matrix representation of a VZ $2$-group $G$ of order $16$ which affords the character $\Omega(\chi_{\mu})$, where $\chi_{\mu}\in \nl(G)$ (defined in \eqref{VZp^4}) \label{t:2}}\\
		\hline
		\textnormal{Group} $G$ & $Z(G)$ & $G'$ & $H$ & $\psi_\mu\in \Irr(H)$ \textnormal{and} $\mu\in \Irr(Z(G)|G{}')$ \\
		\hline
		\endfirsthead
		\hline
		\multicolumn{5}{|c|}{Continuation of Table 1}\\
		\hline
		Group $G$ & $Z(G)$ & $G'$ & $H$ & $\psi$ \\
		\hline
		\endhead
		\hline
		\endfoot
		\hline
		\endlastfoot
		\hline $G_1=\langle x, y, z : x^4=y^2=z^2=1, [x, y]=[x,z]=1, [y, z]=x^2\rangle$ & $\langle x \rangle$ & $\langle x^2 \rangle$ & $\langle x, y \rangle$ & $\psi_\mu(h) = \begin{cases}
			\mu(x)  &\quad \text{ if } h=x,\\
			1  &\quad \text{ if } h=y\\
		\end{cases}$\\
		
		\hline $G_2=\langle x, y : x^8=y^2=1, [x, y]=x^4\rangle$ & $\langle x^2 \rangle$ & $\langle x^4 \rangle$ & $\langle x^2, y \rangle$ & $\psi_\mu(h) = \begin{cases}
			\mu(x^2)  &\quad \text{ if } h=x^2,\\
			\mu(\alpha_3)  &\quad \text{ if } h=y\\
		\end{cases}$\\
		
		\hline $G_3=\langle x, y, z : x^4=y^2=z^2=1, [x, z]=[y,z]=1, [x, y]=x^2\rangle$ & $\langle x^2, z \rangle$ & $\langle x^2 \rangle$ & $\langle x^2, y, z\rangle$ & $\psi_\mu(h) = \begin{cases}
			\mu(x^2)  &\quad \text{ if } h=x^2,\\
			1  &\quad \text{ if } h=y,\\
			\mu(z)  &\quad \text{ if } h=z\\
		\end{cases}$\\
		
		\hline $G_4=\langle x, y, z : x^4=y^2=z^2=1, [x, z]=[y,z]=1, [x, y]=z\rangle$ & $\langle x^2, z \rangle$ & $\langle z \rangle$ & $\langle x^2, y, z \rangle$ & $\psi_\mu(h) = \begin{cases}
			\mu(x^2)  &\quad \text{ if } h=x^2,\\
			1  &\quad \text{ if } h=y,\\
			\mu(z)  &\quad \text{ if } h=z\\
		\end{cases}$\\
		
		\hline $G_5=\langle x, y : x^4=y^4=1, [x, y]=x^2\rangle$ & $\langle x^2, y^2 \rangle$ & $\langle x^2 \rangle$ & $\langle x^2, y \rangle$ & $\psi_\mu(h) = \begin{cases}
			\mu(x^2)  &\quad \text{ if } h=x^2,\\
			(\mu(y^2))^{\frac{1}{2}}  &\quad \text{ if } h=y\\
		\end{cases}$\\
		
		\hline \vtop{\hbox{\strut $G_6=\langle x, y, z : x^4=y^4=z^2=1, [x, z]=[y,z]=1, [x, y]=x^2,$}\hbox{\strut \quad \quad \quad $x^2=y^2\rangle$}} & $\langle x^2, z \rangle$ & $\langle x^2 \rangle$ & $\langle x, z \rangle$ & $\psi_\mu(h) = \begin{cases}
			(\mu(x^2))^{\frac{1}{2}}  &\quad \text{ if } h=x,\\
			\mu(z)  &\quad \text{ if } h=z\\
		\end{cases}$\\
		
		\hline
	\end{longtable}
\end{tiny}
\end{theorem}
\begin{proof}
Let $G$ be a non-abelian group of order $16$ of nilpotency class $2$. Observe that $G$ is a VZ $2$-group. Let $\chi_\mu \in \nl(G)$ as defined in \eqref{VZp^4}. Suppose $(H, \psi_\mu)$ is a special required pair to obtain an irreducible rational matrix representation of $G$ which affords the character $\Omega(\chi_\mu)$. By Proposition \ref{prop:VZreqpair}, it follows that $Z(G) \subset H$ and $\psi_\mu\downarrow_{Z(G)} = \mu$. Further, from Corollary \ref{cor:reqpairVZp5}, $H$ is abelian. In the view of Remark \ref{remark:requiredpair2}, we have following: if $m_\mathbb{Q}(\chi_\mu) =1$ then $\mathbb{Q}(\psi_\mu) = \mathbb{Q}(\chi_\mu)$, and if $m_\mathbb{Q}(\chi_\mu) =2$ then $[\mathbb{Q}(\psi_\mu) : \mathbb{Q}(\chi_\mu)]=2$. Therefore, from Lemma \ref{lemma:fieldofcharacters}, we must choose $\psi_\mu \in \lin(H)$ such that $|\ker(\psi_\mu)| = 2|\ker(\mu)|$ whenever $m_\mathbb{Q}(\chi_\mu) =1$, and $\psi_\mu \in \lin(H)$ such that $|\ker(\psi_\mu)| = |\ker(\mu)|$ whenever $m_\mathbb{Q}(\chi_\mu) =2$. Note that $m_\mathbb{Q}(\chi_\mu)=2$ for $\chi_\mu \in \nl(G_5)$ where $\mu \in \Irr(Z(G_5)|G'_5)$ is given by $\mu(x^2) =-1, \mu(y^2)=-1$, and $m_\mathbb{Q}(\chi_\mu)=2$ for all $\chi_\mu \in \nl(G_6)$, where $\mu \in \Irr(Z(G_6)|G'_6)$. It is routine to check that the pairs $(H,\psi_{\mu})$ mentioned in Table \ref{t:2} are special required pairs. This shows that Table \ref{t:2} presents special required pairs $(H, \psi_\mu)$ to find irreducible rational matrix representations of $G$ whose kernels do not contain $G{}'$, where $G\in \Phi_2$. This completes the proof of Theorem \ref{thm:reqpairVZ16}.
\end{proof}

   \subsection{$p$-groups of order $p^4$ of nilpotency class $3$.} \label{subsec:p4class3}
   Let $G$ ba a $p$-group of order $p^4$ of nilpotency class $3$. Then $|Z(G)| = p$, $|G'| = p^2$ and $Z(G) \subset G{}'$. Further, $|\lin(G)|=p^2$, $|\nl(G)|=p^2 - 1$ and $\cd(G)=\{1,p\}$.
%
  We begin by presenting a few results that enable us to determine all the inequivalent irreducible rational matrix representations of $G$.
   
   \begin{lemma}\label{lemma:fieldextension}
   	Let $G$ be a $p$-group (odd prime $p$), and let $1 \neq \chi \in \Irr(G)$. Then $\mathbb{Q}(\chi) \neq \mathbb{Q}$.
   \end{lemma}
   \begin{proof} Obvious.
   \end{proof}


\begin{lemma}\label{lemma:uniqueabelian}
	Let $G$ be a non-abelian group with nilpotency class $\geq 3$. Suppose that there exists a maximal normal subgroup $H$ of $G$ such that both $H$ and $G/H$ are abelian. Then $H$ is unique.
\end{lemma}
   \begin{proof}
	Since $G/H$ is abelian, $G' \subseteq H$. As $H$ is abelian, we get $C_G(G') \supseteq H$, where $C_G(G')$ denotes the centralizer subgroup of $G'$. Since nilpotency class of $G$ is $\geq 3$,  $G' \nsubseteq Z(G)$. This implies that  $C_G(G') \neq G$. Thus, we conclude that $C_G(G') = H$, which implies the uniqueness of $H$.
	 \end{proof}

\begin{corollary}\label{cor:uniqueabelianp^4}
	If $G$ is a non-abelian group of order $p^4$ of nilpotency class $3$, then $G$ has a unique abelian subgroup of index $p$.
\end{corollary}
\begin{proof}
	  Let $G$ be a non-abelian $p$-group of order $p^4$ of nilpotency class $3$. Then $Z(G) \subset G'$ and there exists an abelian subgroup $H$ of $G$ with index $p$. Therefore, the proof follows from Lemma \ref{lemma:uniqueabelian}. 
\end{proof}

For an odd prime $p$, we follow James' classification of $p$-groups of order $p^4$ (see \cite{RJ}). All the relevant groups in this subsection belong to $\Phi_{3}$ (refer to \cite[Section 4.4]{RJ}).
\begin{remark}\label{remark:requiredHinPhi3}
	\textnormal{Let $G$ be a non-abelian $p$-group (odd prime $p$) of order $p^4$ in $\Phi_3$ (see \cite[Subsection 4.4]{RJ}). Then the description of unique abelian subgroup $H$ of $G$ of index $p$ is as follows.
	\begin{itemize}
		\item If $G=\Phi_3(1^4)$, then $H=\langle \alpha_1, \alpha_2, \alpha_3 \rangle \cong C_p \times C_p \times C_p$.
		\item If $G=\Phi_3(211)a$, then $H=\langle \alpha^p, \alpha_1, \alpha_2 \rangle \cong C_p \times C_p \times C_p$.
		\item If $G=\Phi_3(211)b_r$ ($r=1, \nu$), then $H=\langle \alpha_1, \alpha_2\rangle \cong C_{p^2} \times C_p$.
\end{itemize}}
\end{remark}

   \begin{lemma}\label{lemma:non-linearcharacternonVZp^4}
   	Let $G$ be a non-abelian $p$-group of order $p^4$ of nilpotency class $3$. Let $H$ be the unique abelian subgroup of $G$ of index $p$. If $\psi \in \Irr(H | G')$, then $\psi^G \in \nl(G)$. Further, for $\chi \in \nl(G)$, there exists some $\psi \in \Irr(H | G')$ such that $\chi = \psi^G$.
   \end{lemma}
\begin{proof}
Consider $\psi \in \Irr(H | G')$. On the contrary, suppose that, $\psi^G \notin \nl(G)$. This implies that $\psi^G$ is a sum of some linear characters of $G$. This implies that $G' \subseteq \ker(\psi^G)\subseteq \ker(\psi)$, which is a contradiction.\\
Now, let $\psi \in \Irr(H | G')$. Then $\psi^G \in \nl(G)$, and hence the inertia group $I_G(\psi)$ of $\psi$ in $G$ is equal to $H$ \cite[Problem 6.1]{I}. Furthermore, ${\psi^{G}}\downarrow_{H} = \sum_{i=1}^{p} \psi_{i}$, where $\psi_{i}$'s are conjugates of $\psi$ in $G$ and $p = |G/I_{G}(\psi)|$. Hence, there are $p$ conjugates of $\psi$ and observe that $\psi^G=\psi_i^G\in \nl(G)$ for each $i$. Thus $|\nl(G)|=\frac{|\Irr(H|G')|}{p} = p^2-1$. This complete the proof of Lemma \ref{lemma:non-linearcharacternonVZp^4}.
\end{proof}
  
  Theorem \ref{thm:reqpairnonVZp^4} provides the necessary information to determine all inequivalent irreducible rational matrix representations for all non-abelian $p$-groups of order $p^4$ of nilpotency class $3$, where $p$ is an odd prime.
\begin{theorem}\label{thm:reqpairnonVZp^4}
	Let $G$ be a non-abelian $p$-group (odd prime $p$) of order $p^4$ belongs to $\Phi_3$. Suppose $H$ is the unique abelian subgroup of $G$ with index $p$ and $\chi \in \nl(G)$ such that $\chi = \psi^G$ for some $\psi \in \Irr(H | G')$. Then $(H, \psi)$ is a required pair to determine an irreducible rational matrix representation of $G$ which affords the character $\Omega(\chi)$. 
\end{theorem}
\begin{proof}
	If $G = \Phi_3(1^4)$, or $\Phi_3(211)a$, then $H \cong C_p \times C_p \times C_p$ (see Remark \ref{remark:requiredHinPhi3}). Suppose $\chi \in \nl(G)$ such that $\chi = \psi^G$ for some $\psi \in \Irr(H | G')$. Then, $\mathbb{Q}(\psi)=\mathbb{Q}(\zeta_p)$.  Observe that $\mathbb{Q}(\chi)=\mathbb{Q}(\psi^G) \subseteq \mathbb{Q}(\psi)$. From Lemma \ref{lemma:fieldextension}, we get $\mathbb{Q}(\chi) =\mathbb{Q}(\psi) = \mathbb{Q}(\zeta_p)$. Next, if $G = \Phi_3(211)b_r$ ($r= 1, \nu$), then $H = \langle \alpha_1, \alpha_2\rangle \cong C_{p^2} \times C_p $ (see Remark \ref{remark:requiredHinPhi3}). Again suppose that $\chi \in \nl(G)$ such that $\chi = \psi^G$ for some $\psi \in \Irr(H | G')$. Then $\mathbb{Q}(\psi) = \mathbb{Q}(\zeta_p)$, or $\mathbb{Q}(\zeta_{p^2})$. If $\mathbb{Q}(\psi) = \mathbb{Q}(\zeta_p)$, then from Lemma \ref{lemma:fieldextension}, $\mathbb{Q}(\psi) = \mathbb{Q}(\chi) = \mathbb{Q}(\zeta_p)$. \\
	Now, suppose $\mathbb{Q}(\psi) = \mathbb{Q}(\zeta_{p^2})$. Observe that
	$\psi(\alpha_1) = \zeta_{p^2}$.
	 Assume that $G = \bigcup \alpha^iH$ $(0 \leq i \leq p-1)$. Then
	 \begin{align*}
	 	\psi^G(\alpha_1)= &\sum_{i=1}^{p-1}{\psi}^{\circ}(\alpha^{-i}\alpha_1\alpha^i), ~ \text{where} ~ {\psi}^{\circ}(g) = \begin{cases}
	 		\psi(g)  &\quad \text{ if } g\in H,\\
	 		0  &\quad \text{ if } g \notin H\\
	 	\end{cases}\\
 	= &\psi(\alpha_1) + \psi(\alpha_1\alpha_2)+ \psi(\alpha_1^{1+p}\alpha_2^2)+\psi(\alpha_1^{1+3p}\alpha_2^3)+ \cdots + \psi(\alpha_1^{1+\frac{(p-1)(p-2)}{2}p}\alpha_2^{p-1})\\
 	= &\psi(\alpha_1) [1 + \psi(\alpha_2)+ \psi(\alpha_1^{p}\alpha_2^2)+\psi(\alpha_1^{3p}\alpha_2^3)+ \cdots + \psi(\alpha_1^{\frac{(p-1)(p-2)}{2}p}\alpha_2^{p-1})]\\
 	= & \theta\zeta_{p^2}, \text{ for some } 0\neq \theta \in \mathbb{Q}(\zeta_p). 
	 \end{align*}
	Therefore, $\mathbb{Q}(\psi) =\mathbb{Q}(\psi^G) = \mathbb{Q}(\zeta_{p^2})$. Hence, $(H,\psi)$ is a required pair. This completes the proof. 
\end{proof}

In Proposition \ref{prop:countingnonVZp^4}, the counting of rational irreducible representations of all $p$-groups (odd prime $p$) of order $p^4$ in $\Phi_3$ is described.
\begin{proposition}\label{prop:countingnonVZp^4}
	Let $G$ be a non-abelian $p$-group (odd prime $p$) of order $p^4$ in $\Phi_3$. Then we have the following. 
	\begin{enumerate}
	\item If $G=\Phi_3(211)a$, or $\Phi_3(1^4)$, then $|\Irr_{\mathbb{Q}}^{(1)}(G)|=1$, $|\Irr_{\mathbb{Q}}^{(\phi(p))}(G)|=p+1$, and $|\Irr_{\mathbb{Q}}^{(\phi(p^2))}(G)|=p+1$.
	\item If $G=\Phi_3(211)b_r~ (r=1, \nu)$, then $|\Irr_{\mathbb{Q}}^{(1)}(G)|=1$, $|\Irr_{\mathbb{Q}}^{(\phi(p))}(G)|=p+1$, $|\Irr_{\mathbb{Q}}^{(\phi(p^2))}(G)|=1$, and $|\Irr_{\mathbb{Q}}^{(\phi(p^3))}(G)|=1$.
	\end{enumerate}
\end{proposition}
\begin{proof}
	Suppose $\chi \in \Irr(G)$ and $E(\chi)$ denotes the Galois conjugacy class of $\chi$ over $\mathbb{Q}$. Then the degree of the rational representation affording the character $\Omega(\chi)$ is $|E(\chi)|\chi(1)$. 
	\begin{itemize}
		\item[(1)] Let $G=\Phi_3(211)a = \langle \alpha, \alpha_1, \alpha_2, \alpha_3 : [\alpha_1, \alpha]=\alpha_2, [\alpha_2, \alpha]=\alpha^p=\alpha_3, \alpha_1^p=\alpha_2^p=\alpha_3^p=1 \rangle$. Then $Z(G)= \langle \alpha^p \rangle \cong C_p$, $G'=\langle \alpha^p, \alpha_2 \rangle \cong C_p \times C_p$, and $G/G'=\langle \alpha G', \alpha_1 G' \rangle \cong C_p \times C_p$. Thus, there are 1 Galois conjugacy class over $\mathbb{Q}$ of size 1, and $\frac{p^2-1}{\phi(p)}=p+1$ many distinct Galois conjugacy classes over $\mathbb{Q}$ of size $\phi(p)$ in $\Irr^{(1)}(G)$. Further, $H= \langle \alpha^p, \alpha_1, \alpha_2 \rangle \cong C_p \times C_p \times C_p$ is the abelian subgroup of $\Phi_3(211)a$ of index $p$. Suppose $\chi \in \Irr^{(p)}(G)$. Then from Theorem \ref{thm:reqpairnonVZp^4}, $\chi = \psi^G$ for some  $\psi \in \Irr(H | G')$ and $\mathbb{Q}(\chi)= \mathbb{Q}(\psi)$. Moreover, from Lemma \ref{lemma:fieldextension}, $\mathbb{Q}(\chi)= \mathbb{Q}(\psi)=\mathbb{Q}(\zeta_p)$. Thus, there are $\frac{p^2-1}{\phi(p)}=p+1$ many distinct Galois conjugacy classes over $\mathbb{Q}$ of size $\phi(p)$ in $\Irr^{(p)}(G)$. We get similar results for $G=\Phi_3(1^4)$. This completes the proof of  the part (1) of Proposition \ref{prop:countingnonVZp^4}.
		\item[(2)] Let $G=\Phi_3(211)b_1= \langle \alpha, \alpha_1, \alpha_2, \alpha_3 : [\alpha_1, \alpha]=\alpha_2, [\alpha_2, \alpha]=\alpha_1^p=\alpha_3, \alpha^p=\alpha_2^p=\alpha_3^p=1 \rangle$. Then $Z(G)=\langle \alpha_1^p \rangle = C_p$, $G'=\langle \alpha_1^p, \alpha_2 \rangle = C_p \times C_p$, and $G/G'=\langle \alpha G', \alpha_1 G' \rangle \cong C_p \times C_p$. Thus, there are 1 Galois conjugacy class over $\mathbb{Q}$ of size 1, and $\frac{p^2-1}{\phi(p)}=p+1$ many distinct Galois conjugacy classes over $\mathbb{Q}$ of size $\phi(p)$ in $\Irr^{(1)}(G)$. Further, $H= \langle \alpha_1, \alpha_2 \rangle \cong C_{p^2} \times C_p$ is the abelian subgroup of $\Phi_3(211)b_1$ of index $p$. Suppose $\chi \in \Irr^{(p)}(G)$. Then from Theorem \ref{thm:reqpairnonVZp^4}, $\chi = \psi^G$ for some  $\psi \in \Irr(H | G')$ and $\mathbb{Q}(\chi)= \mathbb{Q}(\psi)$. Observe that there are total $p^2-p$ many $\psi \in \Irr(H | G')$ such that $\mathbb{Q}(\psi) = \mathbb{Q}(\zeta_p)$, and there are total $p^3-p^2$ many $\psi \in \Irr(H | G')$ such that $\mathbb{Q}(\psi) = \mathbb{Q}(\zeta_{p^2})$. Since there are $p$ conjugates of each $\psi \in \Irr(H | G')$, the number of complex irreducible characters of degree $p$ $(\chi \in \Irr^{(p)})$ such that $\mathbb{Q}(\chi)=\mathbb{Q}(\zeta_p)$, and $\mathbb{Q}(\chi)=\mathbb{Q}(\zeta_{p^2})$ are $\frac{p^2-p}{p}=p-1$, and $\frac{p^3-p^2}{p}=p^2-p$ respectively. Thus, there are $\frac{p-1}{\phi(p)}=1$ Galois conjugacy class over $\mathbb{Q}$ of size of size $\phi(p)$, and $\frac{p^2-p}{\phi(p^2)}=1$ Galois conjugacy class over $\mathbb{Q}$ of size $\phi(p^2)$ in $\Irr^{(p)}(G)$. We get similar results for $G=\Phi_3(211)b_\nu$. This completes the proof of the part (2) of Proposition \ref{prop:countingnonVZp^4}.    \qedhere
	\end{itemize}
\end{proof}
 
 Now we prove Theorem \ref{thm:wedderburn nonVZp^4}, which provides the Wedderburn decomposition of $\mathbb{Q}G$, where $G$ is a $p$-group (odd prime $p$) of order $p^4$ of nilpotency class $3$.
\begin{proof}[Proof of Theorem \ref{thm:wedderburn nonVZp^4}]
Observe that $G\in \Phi_3$. Suppose $\chi\in \nl(G)$ and $H$ is the unique abelian subgroup of $G$. Then from Lemma \ref{lemma:non-linearcharacternonVZp^4}, there exists $\psi \in \Irr(H | G')$ such that $\chi = \psi^G$, and if $\psi \in \Irr(H | G')$ then $\psi^G \in \nl(G)$. Now by Theorem \ref{thm:reqpairnonVZp^4}, $\mathbb{Q}(\chi)=\mathbb{Q}(\psi)$. Observe that $\chi^{\sigma}=(\psi^{\sigma})^G$, where $\sigma \in \gal(\mathbb{Q}(\chi)/\mathbb{Q})$ and there are exactly $p$ distinct conjugates $\psi\in \Irr(H|G{}')$ such that $\chi=\psi^G$ (see the proof of Lemma \ref{lemma:non-linearcharacternonVZp^4}). Let $X$ and $Y$ be the representative set of distinct Galois conjugacy classes of $\Irr(G)$ and $\Irr(H)$, respectively. Let $d$ be a divisor of $\exp(H)$ such that $\mathbb{Q}(\chi)=\mathbb{Q}(\psi) = \mathbb{Q}(\zeta_d)$. Set $m=\exp(H)$ and $m{}'=\exp(H/G{}')$. Then we have two cases.\\
{\bf Case 1 ($d \mid m$ but $d \nmid m'$).} In this case, we have
\begin{align*}
|\{\chi \in X : \chi(1)=p, \mathbb{Q}(\chi)=\mathbb{Q}(\zeta_d) \}|&=\frac{1}{p}|\{\psi \in Y : \psi\in \Irr(H|G{}'), ~\mathbb{Q}(\psi)=\mathbb{Q}(\zeta_d) \}|\\
&= \frac{a_d}{p}, 
\end{align*}
where $a_d$ denotes the number of cyclic subgroups of order $d$ of $H$ (see Lemma \ref{lemma:Ayoub}). \\
{\bf Case 2 ($d \mid m$ and $d \mid m'$).} In this case, we have
\begin{align*}
	|\{\chi \in X : \chi(1)=p, \mathbb{Q}(\chi)=\mathbb{Q}(\zeta_d) \}|&=\frac{1}{p}|\{\psi \in Y : \psi\in \Irr(H|G{}'), ~\mathbb{Q}(\psi)=\mathbb{Q}(\zeta_d) \}|\\
	&= \frac{a_d-a'_d}{p}, 
\end{align*}
where $a_d$ and $a'_d$ denote the number of cyclic subgroups of order $d$ of $H$ and $H/G{}'$, respectively (see Lemma \ref{lemma:Ayoub}).\\
Now, let $A_{\mathbb{Q}}(\chi)$ be the simple component of the Wedderburn decomposition of $\mathbb{Q}G$ corresponding to the rational representation of $G$ that affords the character $\Omega(\chi)$. Then $A_{\mathbb{Q}}(\chi) \cong M_n(D)$ for some $n \in \mathbb{N}$ and a division ring $D$. Observe that $n=p$ and $D = \mathbb{Q}(\chi)$ (see Lemma \ref{Reiner}). Therefore all the irreducible rational 
representations of $G$ whose kernel do not contains $G{}'$ will contribute
$$\bigoplus_{d|m, d \nmid m'} \frac{a_{d}}{p}M_p(\mathbb{Q}(\zeta_{d}))\bigoplus_{d|m, d|m'}\frac{a_{d}-a_{d}'}{p}M_p(\mathbb{Q}(\zeta_{d}))$$
in the Wedderburn decomposition of $\mathbb{Q}G$. This completes the proof of Theorem \ref{thm:wedderburn nonVZp^4}.
\end{proof}
 Corollary \ref{cor:wedderburn nonVZp^4} immediately follows from Theorem \ref{thm:wedderburn nonVZp^4}.
\begin{corollary}\label{cor:wedderburn nonVZp^4}
		Let $G$ be a non-abelian $p$-group (odd prime $p$) of order $p^4$ in $\Phi_3$.
		\begin{itemize}
			\item[(a)] If $G=\Phi_3(211)a$, or $\Phi_3(1^4)$, then 
			\[\mathbb{Q}G \cong \mathbb{Q} \bigoplus (p+1)\mathbb{Q}(\zeta_p) \bigoplus (p+1)M_p(\mathbb{Q}(\zeta_p)).\]
			\item[(b)] If $G=\Phi_3(211)b_r ~ (r = 1,\nu)$, then 
			\[\mathbb{Q}G \cong \mathbb{Q} \bigoplus (p+1)\mathbb{Q}(\zeta_p) \bigoplus M_p(\mathbb{Q}(\zeta_p)) \bigoplus M_p(\mathbb{Q}(\zeta_{p^2})).\]
		\end{itemize}
\end{corollary}

We end this subsection with the following remark.
\begin{remark}\label{rem:nonvz16} \textnormal{It is well known that if $G$ is a $2$-group of maximal class then $G$ is isomorphic to one of the following: dihedral group, semi-dihedral group, or generalized quaternion group. Let $G$ be a $2$-group of order $16$ of maximal class (i.e, of nilpotency class $3$). Then $G$ has two inequivalent irreducible $2$-dimensional faithful complex representations and one irreducible $2$-dimensional non-faithful complex representation. Observe that the irreducible $2$-dimensional non-faithful complex representation of $G$ is realizable over $\mathbb{Q}$. From Lemma \ref{lemma:rationalQDSd}, one can compute an irreducible rational matrix representation of $G$ which affords the character $\Omega(\chi)$, where $\chi \in \FIrr(G)$.}
\end{remark}

	\section{Acknowledgements}
	Ram Karan acknowledges University Grants Commission, Government of India. The corresponding author acknowledges SERB, Government of India for financial support through grant (MTR/2019/000118).

\end{document}